\documentclass[a4paper, 11pt]{amsart}
\usepackage[utf8]{inputenc}
\usepackage{mathrsfs}
\usepackage{amsfonts, amssymb, amsxtra}
\usepackage{amsmath, bm}
\usepackage{dsfont}
\usepackage{color}
\usepackage{graphicx}
\usepackage[english,polish]{babel}
\usepackage{enumerate}
\usepackage{mathtools}
\usepackage{bbm}
\usepackage{enumitem}

\usepackage[T1]{fontenc}

\usepackage[compress, sort]{cite} 

\usepackage{newtxmath}
\usepackage{newtxtext}

\usepackage[margin=2.5cm, centering]{geometry}
\usepackage[colorlinks,citecolor=blue,urlcolor=blue,bookmarks=true]{hyperref}

\usepackage[utf8]{inputenc}

\newcommand{\CC}{\mathbb{C}}
\newcommand{\ZZ}{\mathbb{Z}}

\newcommand{\NN}{\mathbb{N}}
\newcommand{\RR}{\mathbb{R}}

\newcommand{\xx}{\mathbbm{x}}

\newcommand{\calC}{\mathcal{C}}
\newcommand{\calL}{\mathcal{L}}

\newcommand{\calB}{\mathcal{B}}

\newcommand{\calH}{\mathcal{H}}

\newcommand{\calJ}{\mathcal{J}}

\newcommand{\frakM}{\mathfrak{M}}

\newcommand{\JLf}[1]{\ell_{#1}}

\newcommand{\Id}{\operatorname{I}}

\newcommand{\norm}[1]{\left\lVert {#1} \right\rVert}

\newcommand{\lnorm}[1]{\downharpoonleft\!\!\!| {#1} |\!\!\!\downharpoonright}

\newcommand{\HSnorm}[1]{\| {#1} \|_{\mathrm{HS}}}

\newcommand{\sprod}[2]{\left\langle {#1}, {#2} \right\rangle}

\newcommand{\tr}{\operatorname{tr}}

\newcommand{\sigmaP}{\sigma_{\mathrm{p}}}

\newcommand{\sigmaAC}{\sigma_{\mathrm{ac}}}

\newcommand{\muAC}{{\mu_{\mathrm{ac}}}}
\newcommand{\muS}{{\mu_{\mathrm{sing}}}}
\newcommand{\muSC}{{\mu_{\mathrm{sc}}}}
\newcommand{\muPP}{{\mu_{\mathrm{pp}}}}
\newcommand{\Msing}{{M_{\mathrm{sing}}}}
\newcommand{\Mac}{{M_{\mathrm{ac}}}}

\newcommand{\Sac}{S_{\mathrm{ac}}}
\newcommand{\Sacr}{S_{\mathrm{ac}, \mathrm{r}}}

\newcommand{\Ssing}{S_{\mathrm{sing}}}

\newcommand{\Jmin}{J_\mathrm{min}}

\newcommand{\Dom}{\mathrm{Dom}}
\newcommand{\rank}{\mathrm{rank}}

\newcommand{\ud}{{\: \rm d}}

\newcommand{\cl}[1]{\operatorname{cl}(#1)}
\newcommand{\clleb}[1]
{\overline{#1}^{_\mathrm{e}}}

\renewcommand{\Im}{\operatorname{Im}}
\renewcommand{\Re}{\operatorname{Re}}

\newtheorem{theorem}{Theorem}
\newtheorem{lemma}[theorem]{Lemma}
\newtheorem{proposition}[theorem]{Proposition}
\newtheorem{fact}[theorem]{Fact}
\newtheorem{corollary}[theorem]{Corollary}

\newtheorem{observation}[theorem]{Observation}
\theoremstyle{definition}
\newtheorem{example}[theorem]{Example}

\newtheorem{definition}[theorem]{Definition}

\numberwithin{equation}{section}
\numberwithin{theorem}{section}

\DeclareMathOperator{\lin}{lin}
\DeclareMathOperator{\Ini}{Ini}

\newcommand{\bde}{\begin{description}}
\newcommand{\ede}{\end{description}}

\newcommand{\bite}{\begin{itemize}}
\newcommand{\eite}{\end{itemize}}

\newcommand{\beql}[1]{\begin{equation}\label{#1}}
\newcommand{\eeq}{\end{equation}}

\newcommand{\kwa}{$\Box$}
\newcommand{\Kwa}{\hfill\kwa\bigskip}

\newenvironment{proofof}[1]{\par\vspace{0.7em}\noindent{\it Proof of #1.}}{\Kwa\hspace{1ex}}

\newcommand{\strz}{\longrightarrow}

\newcommand{\opcje}[4]{\left\{
    \begin{array}{ll}#1&\mbox{for $#2$}\\#3&\mbox{for $#4$}\end{array}\right.}

\newcommand{\N}{\NN}
\newcommand{\Nk}[1]{\N_{#1}}  
\newcommand{\Nz}{\Nk{0}}    
\newcommand{\Nmj}{\Nk{-1}}

\newcommand{\eps}{\epsilon}

\newcommand{\si}{\sigma}

\newcommand{\la}{\lambda}

\newcommand{\de}{\delta}

\newcommand{\om}{\omega}

\newcommand{\Mdc}{M_{d}(\CC)}
\newcommand{\Mddc}{M_{2d}(\CC)}

\newcommand{\lesc}{\left\langle} 
\newcommand{\risc}{\right\rangle}

\newcommand{\row}[2]{#1_{\{#2\}}} 
\newcommand{\col}[2]{#1^{\{#2\}}} 

\newcommand{\el}{\ell}
\newcommand{\lnz}[1]{\el(\Nz, #1)}
\newcommand{\lnmj}[1]{\el(\Nk{-1}, #1)}
\newcommand{\lnk}[2]{\el(\Nk{#1}, #2)}
\newcommand{\lfinnz}[1]{\el_{\text{fin}}(\Nz, #1)}
\newcommand{\lfinnk}[2]{\el_{\text{fin}}(\Nk{#1}, #2)}
\newcommand{\ldnz}[1]{\el^2(\Nz, #1)}
\newcommand{\ldnmj}[1]{\el^2(\Nk{-1}, #1)}

\newcommand{\sprz}[1]{\overline{#1}}
\newcommand{\Bor}{\mathrm{Bor}}
\newcommand{\restr}[2]{#1\hspace{-1.0ex}  \restriction_{#2}}

\newcommand{\asympp}[1]{\asymp_{#1}}

\newcommand{\Ker}{\operatorname{Ker}}
\newcommand{\Ran}{\operatorname{Ran}}

\newcommand{\Forall}{\mbox{\Large\mbox{$\forall$}}}
\newcommand{\Exists}{\mbox{\Large\mbox{$\exists$}}}

\newcommand{\nic}[1]{} 
\newcommand{\af}[1]{\mathrm{aff}\left(#1\right)}

\newcommand{\BLW}{L(W)}

\newcommand{\GEV}[1]{\mathrm{GEV}_{-1}\hspace{-0.1em}\left(#1\right)}
\newcommand{\MGEV}[1]{\mathrm{MGEV}_{-1}\hspace{-0.1em}\left(#1\right)}

\newcommand{\GEVze}[1]{\mathrm{GEV}\hspace{-0.1em}\left(#1\right)}
\newcommand{\MGEVze}[1]{\mathrm{MGEV}\hspace{-0.1em}\left(#1\right)}

\newcommand{\GGEVze}{\mathrm{GEV}}
\newcommand{\MMGEVze}{\mathrm{MGEV}}

\newcommand{\EV}[1]{\mathrm{EV}\hspace{-0.2em}\left(#1\right)}

\newcommand{\GEVld}[1]{\mathrm{GEV}_{\ell^2}\hspace{-0.2em}\left(#1\right)}

\newcommand{\Para}{\mathrm{Par}}

\newcommand{\extl}[1]{_{\bullet}#1}
\newcommand{\extlz}[2]{_{#1\bullet}#2}

\newcommand{\Caj}{{\calC}_J}
\newcommand{\Wej}{W}

\newcommand{\DD}{\mathbb{D}}

\newcommand{\nnad}{\overline{n}}
\newcommand{\npod}{\underline{n}}

\title{Block Jacobi matrices and Titchmarsh--Weyl function}
\author{Marcin Moszy\'{n}ski}
\address{
    Marcin Moszy\'{n}ski \\
    Faculty of Mathematics, Informatics and Mechanics \\
    University of Warsaw \\
    ul. Stefana Banacha 2 \\
    02-097 Warsaw, Poland
}
\email{mmoszyns@mimuw.edu.pl}

\author{Grzegorz \'{S}widerski}
\address{
    Grzegorz \'{S}widerski \\
    Institute of Mathematics \\ 
    Polish Academy of Sciences \\
    ul. Śniadeckich 8 \\
    00-696 Warsaw, Poland
}
\email{grzegorz.swiderski@pwr.edu.pl}
\curraddr{Faculty of Pure and Applied Mathematics, Wroclaw University of Science and Technology, Wyb. Wyspiańskiego 27, 50-370 Wroclaw, Poland}


\begin{document}
\selectlanguage{english}

\begin{abstract}
We collect some results and notions concerning  generalizations for  block Jacobi matrices  of several concepts, which have been important for spectral studies of  the simpler and better known scalar Jacobi case. 
We focus here on some  issues related to the matrix Titchmarsh--Weyl function, but we also consider generalizations of some other tools used by  subordinacy theory, including the  
matrix 
 orthogonal polynomials, the notion of finite cyclicity, a variant of a  notion of nonsubordinacy, as well as 
Jitomirskaya--Last type semi-norms.

The article brings together some issues already known,  our new concepts, and also improvements and strengthening of some results  already existing.
 We  give simpler  proofs of some  known facts or  we add   details usually  omitted in the existing literature. 
The introduction contains a separate part devoted to a brief review of the main spectral analysis methods used so far for block Jacobi operators. 

\end{abstract}

\maketitle
\tableofcontents

\section{Introduction}

 Jacobi matrices (JM) and the appropriate Jacobi operators (JO) in  the standard  Hilbert space $\ell^2$ of  square summable scalar  sequences  
  have been classical objects of interest to mathematicians for years, especially to those involved in spectral analysis. Therefore, the mathematical literature on JO is currently very rich.

 Block Jacobi matrices (BJM) and operators (BJO) are    generalizations of those scalar ones with the scalar coefficients replaced by square matrices. 
 They  were originally introduced by Kre{\u \i}n in \cite{Krein1949a, Krein1949b} in  1940s, 
 but the number of strict mathematical results for them  is still small, compared to those for  scalar JM-s and JO-s.

 Their  basic properties are discussed  in a monograph of Berezanski{\u \i} \cite[Section VII.§2]{Berezanskii1968}. Other good references are \cite{Damanik2008} and \cite[Chapter 12 and 13]{Bateman1}.

The interest in block Jacobi operators comes from their close  relation to the matrix moment problem (see e.g. \cite{Duran2001a, Berg2008}) as well as from the theory of matrix orthogonal polynomials on the real line, see e.g. \cite{Damanik2008}. 
They are also  useful for analysis of difference equations of finite order, see \cite{Duran1995}. 
Some types of block Jacobi operators are related to random walks and level dependent quasi--birth--death processes, see e.g. \cite{Dette2007}. 
 In \cite{Aptekarev1984} BJO found some applications in mathematical physics. Some extensions of BJO are useful in studying the scalar periodic JO (see e.g. \cite[Chapter 8]{Simon2010Book}), various moment problems (see \cite{Berezanskii2023}) and orthogonal polynomials of several variables (see \cite[Section 3.4]{Dunkl2014}). 
 Finally, let us mention that BJO on $\ZZ$ can be conveniently modeled as BJO on $\NN_0$ by doubling the dimension of the blocks (see e.g. \cite[Chapter VII.§3]{Berezanskii1968}).
For further applications we refer to \cite{Sinap1996}.

Before discussing the subject of this paper, we present here a brief overview of several spectral analysis methods for BJO used so far.

\subsection{Short review on spectral analysis methods for block Jacobi Operators}

The spectral analysis of block Jacobi operators has already been partially developed in several directions, using methods that are, in general, adaptations (often far from being trivial) of methods previously used for JO.

Let's mention first some non-strictly spectral studies, related to the selfadjointness  problem.
Papers \cite{Kostyuchenko1998} and \cite{Dyukarev2020}  present  results corresponding to the maximal possible deficiency indices for the so-called \emph{completely indeterminate} case.
Moreover,  in \cite{Dyukarev2006} it is shown that minimal block Jacobi operators can have non-equal deficiency indices.

Below we briefly sketch the main  approaches to spectral studies of BJO  in the mathematical literature, organized  according to the type of the methods used.

\subsubsection{Perturbations theory for block Toeplitz operators}

It is one of the first  methods which has been applied to get some  ''strictly spectral'' results for  BJO-s. 
Its goal  is  to control the number of eigenvalues in  spectral gaps (open intervals outside  the essential spectrum) of the operator.
The method consists in using 
some   abstract perturbation theory results (see Cojuhari: \cite[Sections 1--3]{Cojuhari1981} and  \cite[Section 2]{Cojuhari2009})  to study  perturbations of 
 matrix Toeplitz operators, called also matrix Wiener--Hopf discrete operators (see e.g. the monograph \cite{BottcherSilbermann2006}). 
Recall that BJO with constant coefficients are  bounded self-adjoint matrix Toeplitz operators
with a simple  formula for the symbol  (see, e.g., \cite[Section 2]{Cojuhari2007}), so their spectral properties could be precisely described thanks to the  general theory (see e.g. \cite{GohbergKrein1958}).
The above general idea is a ``matrix adaptation'' of  previous 
 ''scalar version'' used for  JO (see, e.g. \cite{Cojuhari2004}). 
The method can be used  for ``small perturbations'' of BJO-s with constant blocks --- see, e.g.,  \cite{Cojuhari2007}. 
The results of \cite{Cojuhari1981, Cojuhari2009}
 refer to perturbations of {\bf scalar} JO with periodic coefficients, which can be treated as the appropriate BJO with constant blocks. In such a case, if the period is at least two, the off-diagonal blocks are singular, so the methods using transfer matrices for BJO (e.g.  in  the present  paper, where we assume \eqref{zal-odw+sym}) are useless in this case.

\subsubsection{Combes--Thomas type estimates}

It is well-known (see, e.g. \cite[Corollary 2.4]{Shubin1985}) that the entries of the resolvent of \emph{bounded} BJO decay exponentially with respect to the distance between the entries. Since~a similar conclusion was earlier shown by Combes--Thomas in \cite{Combes1973} for the resolvent of multiparticle system Schr\"{o}dinger operators, estimates of this kind for the resolvent are sometimes called \emph{Combes--Thomas estimates} (see, e.g. \cite[Proposition 2.3]{Breuer2016}, \cite[Theorem A.1]{Breuer2010}).

By extending the method of Combes--Thomas in the articles Janas--Naboko--Stolz \cite{JanasNabokoStolz2009} and Janas--Naboko \cite{Janas2013} some analogues of such estimates for possibly \emph{unbounded} JO were established for some regions of the resolvent set. Later in Janas--Naboko--Silva \cite{Janas2018,Janas2020} these methods were adapted to BJO. Finally, in Naboko--Simonov \cite{NabokoSimonov2023}, such estimates for BJO were proved for the whole resolvent set.

\subsubsection{Estimates of quadratic forms}
Estimates of the quadratic form of JO have been a popular topic in the literature, see e.g. \cite{DombrowskiJanasMoszynskiEtAl2004, JanasNabokoStolz2004, DombrowskiPedersen2002, DombrowskiPedersen2002a, Dombrowski2004, Dombrowski2009, Szwarc2002, NabokoJanas2003}. Such estimates allowed to discover gaps in the essential spectrum of some classes of JO. In the article Kupin--Naboko \cite{Kupin2018}, by extending  some methods of \cite{Damanik2007}, some estimates of the quadratic form for a quite general class of BJO were established. 
In a different style, in \cite{block2018}, by extending the techniques developed for JO from \cite{Monotonic, PeriodicI}, estimates for some quadratic forms of two consecutive indices of generalized eigenvectors (the so-called Tur\'{a}n determinants) were established for a wide class of BJO. As a consequence some asymptotic estimates for generalized eigenvectors were established leading to continuity of the spectrum. Note that, at the same time, this asymptotic information   turned out to be  crucial to get important example for the main result of  \cite{BarrierNonsubordinacy}.

\subsubsection{Commutator methods}
For scalar JO various commutator methods are quite popular, see e.g. \cite{Dombrowski2011, Pedersen2002, Monotonic} (extensions of Putnam--Kato's method), \cite{Sahbani2008} (Mourre's method) and \cite{Monvel2018} (a double commutator method).

In Janas \cite[Theorem 3]{Janas2014} emptiness of the point spectrum of some BJO was established by expanding  \cite{Pedersen2002}. Later, in \cite[Theorem 6]{block2018}, this result was partially generalized  and the continuous spectrum was identified. 
In Sahbani \cite{Sahbani2016} the Mourre's commutator method was applied to study compact perturbations of constant coefficients for BJO.  Under some regularity hypotheses it was shown there that the singular continuous spectrum is empty.

\subsubsection{Compactness of the resolvent}
For scalar JO  it is sometimes possible to prove directly that the resolvent is a compact operator, which leads to the conclusion that the essential spectrum of JO is empty, see e.g. \cite[Section 4]{JanasNaboko2002}, \cite[Theorem 3.1(iv)]{Moszynski2003}, \cite{Janas2007}, \cite[Theorem 2.1]{Sahbani2008}, \cite{Janas2007}. By exploiting a connection with BJO with singular off-diagonals this approach was also implemented in \cite[Theorem B(c)]{PeriodicIII}. 
In \cite{Budyka2022} compactness of the resolvent for a class of BJO was studied.

\subsubsection{Matrix Titchmarsh--Weyl Function and subordinacy methods}

\vspace{6ex}

The scalar Titchmarsh--Weyl function (called also Weyl--Titchmarsh or simply Weyl function) has been  one of  useful tools for spectral studies of some differential (see, e.g. the monograph \cite{Behrndt2020}) and difference operators (the ''scalar'' ones), including Jacobi operators. The  Weyl function has been    generalized also for BJO case as the matrix Weyl function 
(see, e.g., \cite[Section VII.§2.10]{Berezanskii1968}, \cite{Acharya2019} and see subsections \ref{subs:Thepresentwork}, \ref{l2mW} for some more explanations), however, its  use in spectral analysis of BJO has been difficult and papers on such applications have only recently begun to appear.

 Some  scalar methods from \cite{Jitomirskaya1999, LastSimon1999} using among other things the Weyl function  has been recently adapted in Oliveira--Carvalho \cite{Oliveira2022} 
 to BJO  case (see also \cite{Oliveira2021a} for  other use of matrix Weyl function in spectral analysis for BJM considered on $\ZZ$).

It seems that the switch from the scalar to the $d>1$ case is especially problematic for the theory of subordinacy   (also see subsection \ref{subs:Thepresentwork}), being one of the main spectral methods based on the Weyl function. As far as we know, there is still  no  understanding of how a complete analog of subordinacy theory for BJO might look like. 
Nevertheless, 
the present  paper contains a kind of introduction to some possible spectral methods based on the matrix Weyl function, including one of   ''non-subordinacy type'' notions  and a simple ''subordinacy theory type'' spectral result for BJO. 
Moreover, in our parallel article \cite{BarrierNonsubordinacy},  
we have  obtained some interesting ''subordinacy-type'' results. Namely, we have found
 a  sufficient condition for absolute continuity of BJO in terms of our new notion of barrier nonsubordinacy.

\subsection{The present work}
\label{subs:Thepresentwork}

This work is devoted to generalizations of several concepts, tools and ''small theories'', which turned out to be important for spectral studies of scalar Jacobi matrices (JM) to the general case of block Jacobi matrices (BJM) with $d\times d$ blocks 
for 
arbitrary finite dimension $d$.
We  focus here especially on 
 the results related to   the matrix   Titchmarsh--Weyl function for BMJ, which will be often named here simply matrix   Weyl function for short, similarly to the scalar function case.
We also pay attention to 
generalizations of some other  tools used by the so-called subordinacy theory (\cite{Gilbert1987, Khan1992,Jitomirskaya1999}). 

The matrix Weyl function  is a  generalization of the scalar Weyl (Titchmarsh--Weyl) function and   was  introduced much earlier for many kinds of self-adjoint operators (see e.g. \cite{Behrndt2020} for a general treatment).  The scalar one was very useful,  e.g.,  for  Jacobi matrices, and from our perspective here, it was  crucial as one of the fundamental objects enabling proofs of main theorems of subordinacy theory for JM. But its history is much longer, see e.g. \cite[Section 8.2]{Everitt2005} and \cite{Behrndt2020}. The matrix Weyl function naturally appeared before in the study of spectral properties of higher order ordinary differential operators, see \cite{Weidmann1987}.

The matrix Weyl function for BJM is closely related to the so-called orthogonal matrix polynomials $P$ and $Q$, to which we also pay some attention here, and which are generalizations of the appropriate scalar orthogonal polynomials $p$ and $q$.

We hope that in the future some of the generalizations considered here will also prove useful in the spectral analysis of BJM. This hope is not groundless: the present work was written in parallel to the paper  \cite{BarrierNonsubordinacy}, where  some of these generalizations have already turned out to be useful exactly  for new spectral results ''of  subordinacy style''  for BJM, which have been missing in spectral studies for BMJ so far.

This article brings together both the results already known, at least in part, and our new results, improvements and strengthening of those already existing.
We extend some known results and sometimes give other proofs, more ingenious and more elementary, while trying to make this work largely self-sufficient.
We add many details often omitted in the existing literature. We also try to somewhat unify  the various notations and terminologies used in the rather scattered literature on this subject, sometimes suggesting slightly different approaches to certain concepts.

Let us now recall some basic issues concerning Jacobi (scalar) matrices.

A \emph{Jacobi matrix} (see, e.g., \cite[Chapter VII.§1]{Berezanskii1968}) is a complex semi-infinite tridiagonal Hermitian matrix of the form
\begin{equation} \label{eq:127}
    \calJ =   
    \begin{pmatrix}
        b_0	& a_0	&		&       \\
        \sprz{a_0}& b_1	& a_1	&         \\
            & \sprz{a_1}	& b_2 	& \ddots   \\
            &		&		\ddots &  \ddots 
    \end{pmatrix}
\end{equation}
with $a_n \not= 0$ and $b_n \in \RR$ for all $n\in\Nz=\{0,\ 1,\ 2, \ldots\}$.
The action of $\calJ$ is well-defined on the linear space  $\lnk{0}{\CC}$ of all complex-valued sequences treated as column vectors, and it is natural to define the  operator $J$ as the ''restriction'' of $\calJ$ to the standard  Hilbert space $\ldnz{\CC}$ of  square summable sequences. Namely, we define $J x = \calJ x$ on the domain consisting of all such  $x \in \ldnz{\CC}$  that also $\calJ x \in \ldnz{\CC}$. Such operator $J$ is called   \emph{the maximal Jacobi operator}, and here we shall usually call it simply \emph{Jacobi operator}\footnote{Note, however, that also some non-maximal operators related to the matrix $\calJ$ are considered, and the name ''Jacobi operator/matrix'' is often used for them, too.} or JO for short. It need not be either bounded or self-adjoint, in general. But if  it is self-adjoint, then one can define a Borel probability measure $\mu$ ({\em the scalar spectral measure of $J$}) on the real line by the formula
\begin{equation} \label{eq:149}
    \mu(G) = \langle E_J(G) \delta_0(1), \delta_0(1) \rangle_{\el^2}, \qquad G \in \Bor(\RR),
\end{equation}
where $E_J$ is the projection-valued spectral measure of $J$ and $\delta_0(1)$ is a sequence with $1$ on $0$th position and $0$ elsewhere. Then one can prove that $J$ is unitary equivalent to the operator acting by multiplication by the identity function on $L^2(\mu)$, see e.g. \cite[Theorem 5.14 and Theorem 6.16]{Schmudgen2017}.

The interest in Jacobi operators comes from their close relation to the classical moment problem as well as the theory of orthogonal polynomials on the real line, see e.g. \cite{Simon1998}. As every self-adjoint operator having a ``non-degenerate'' cyclic vector is unitary equivalent to a maximal or non-maximal Jacobi operator,  those maximal ones  are quite  typical  ''basic building blocks'' of self-adjoint operators. Some types of Jacobi operators are related to random walks and birth--death processes, see e.g. \cite{Karlin1957, Karlin1959}. Finally, Jacobi matrices are very useful in numerical analysis in the construction of Gaussian quadratures, see e.g. \cite{Golub1969}. 

A method of spectral analysis, the {\em theory of subordinacy},   due to Gilbert--Pearson \cite{Gilbert1987} and later, due to Khan--Pearson  in its  Jacobi variant (see \cite{Khan1992}), started to be more and more  prominent during the last three decades. 
Given $\la \in \CC$ a sequence $u = (u_n)_{n \in \Nz}$ 
is called \emph{generalized eigenvector} (associated with $\la$), $u \in \GEVze{\la}$, if it satisfies the recurrence relation
$$
     (\calJ u)_n = \la u_n \quad n \geq 1
$$
with some initial conditions $(u_0,u_1)$. A non-zero sequence $u \in \GEVze{\la}$ is 
\emph{subordinate} if for any linearly independent $v \in \GEVze{\la}$
\begin{equation} \label{eq:130}
    \lim_{n \to \infty} \frac{\norm{u}_{[0,n]}}{\norm{v}_{[0,n]}} = 0,
\end{equation}
where for a sequence $x \in \lnk{0}{\CC}$ and each $n \geq 0$ 
the seminorm $\norm{ \cdot }_{[0,n]}$ is given by
\begin{equation} \label{eq:131}
    \norm{x}_{[0,n]} := \sqrt{\sum_{k=0}^n |x_k|^2}.
\end{equation}
Let us decompose the measure  $\mu$ as \ \ 
$    \mu = \muAC + \muS=\muAC+\muSC+\muPP$, \ \ 
where $\muAC$,  $\muS$, $\muSC$ and  $\muPP$ denote the absolutely continuous,  the singular, the singular continuous and the pure point  part of $\mu$ with respect to the Lebesgue measure.

The main result, \cite[Theorem 3]{Khan1992}, says that the set 
$\Sac = \{ \la \in \RR: \ \text{no non-zero\ } u \in \GEVze{\la} \text{ is subordinate}\}$ is  a minimal support  of  $\muAC$ with respect to the Lebesgue measure
and $\Ssing = \{ \la \in \RR: \ \text{a non-zero } u \in \GEVze{\la} \text{ such that also }\  (\calJ u)_0 = \la u_0 \ \text{ holds,  \  is subordinate}\}$ is a support  of  $\muS$ and its Lebesgue measure is zero.
Since we often have some idea about  asymptotic behaviour of generalized eigenvectors, this theory turned out to be very successful in spectral analysis of various classes of Jacobi matrices, see e.g. \cite{LastSimon1999}.
Note that  similar theories exist also for some other  classes of operators, e.g., for  continuous one-dimensional Sch\"odinger operators on the real half-line (see \cite{Gilbert1987}).

There exist  two main approaches to the subordinacy theory of $J$   in the spectral literature.  The first is the original one,  contained in \cite{Gilbert1987} and \cite{Khan1992}. And the second,   Jitomirskaya--Last's approach (see    \cite{Jitomirskaya1999}),  is  based  on a continuous interpolation
of the discrete  family of semi-norms
\eqref{eq:131}.
As was mentioned earlier, in both approaches,  the key object enabling the transition from subordination issues to the properties of the spectral measure for $J$ is precisely the Weyl function.
It is related to  $\ell^2$ generalized eigenvectors for  $J$ and complex, non-real $\la$-s and to two special non-$\ell^2$ generalized eigenvectors  $p$ and $q$ for   $J$ --- the orthogonal polynomials of the first and of the second kind. 
We recall the detailed definition of the Weyl function in subsection \ref{l2mW},  both  in  the scalar case and in the general matrix case for  BJM with blocks of dimension $d$.  As one can see, the block case  is just the direct generalization of the  definition in   the  scalar case (for MJ), which is simply   the block case with  $d=1$.

The Block Jacobi Matrices and Block Jacobi Operators are  important  generalizations of those scalar ones (see e.g. \cite[Chapter VII.§2]{Berezanskii1968}).
 Recall here that 
BJM is an analog of  JM \eqref{eq:127} with matrix ``blocks'' being its  terms  instead of scalars, namely it is a 
semi-infinite block-tridiagonal Hermitian matrix of the form 
\begin{equation} \label{eq:127'}
    \calJ =   
    \begin{pmatrix}
        B_0	& A_0	&		&       \\
        A_0^*	& B_1	& A_1	&         \\
            & A_1^*	& B_2 	& \ddots   \\
            &		&		\ddots &  \ddots 
    \end{pmatrix},
\end{equation}
where $A_n$ and $B_n$ are $d \times d$ complex matrices with all the $A_n$  invertible and  Hermitian $B_n$. 
And this, together with the induced  maximal Block Jacobi Oparator (BJO)  (its detailed definition is placed in Section \ref{sec:3n}), are the main objects of all our generalizations in this paper.

The article is organized as follows. 

In Section~\ref{sec:2} we fix our notation and we collect some rather elementary but convenient  facts that we  need in subsequent sections.
One of  novelties here (see subsection \ref{J-Lsemi}) is the  analog for the general block case of the $d=1$  Jitomirskaya--Last's  continuous interpolation from \cite{Jitomirskaya1999}
of a discrete  family of semi-norms \eqref{eq:131} from Khan--Pearson \cite{Khan1992} original version of subordinacy theory.

 In Section~\ref{sec:3n} we start from recalling the definition of the  block Jacobi operator  $J$ (the maximal one), and  then we prove the finite-cyclicity of $J$. Note that finite-cyclicity is  a natural generalization  of the notion of  cyclicity, being  the well-known  property of $J$  in the scalar $d=1$ case. 
 In subsection \ref{L2M}, for our main case when  $J$ is self-adjoint,  we recall  the notion of the matrix measure $M$ of $J$.  We also recall 
 the representation of $J$ as the multiplication operator in the 
  $L^2$- matrix measure space $L^2(M)$  of $\CC^d$-valued functions, being an analog of the classical $L^2(\mu)$ Hilbert space
for the  spectral measure $\mu$ of $J$ from the case $d=1$.
Next, we ``compute''  in detail  a general  example \ref{smm-diag} of the matrix measure for the 
``simplest'' block Jacobi operators with  all the   blocks  being   diagonal.

 Section~\ref{sec:3s} is devoted to
  generalized eigenvectors  and to transfer matrices in  the  block case. We study here two types of generalized eigenvectors: ''the usual'', i.e.,  with $\CC^d$ terms,  and the matrix generalized eigenvectors --- with $d\times d$ matrix terms (for the left side multiplication by the blocks of $J$). We give here also a strict approach to their  extensions  to the ''initial values'' at $-1$ and  $0$ instead of $0$ and $1$. 
 Then, in subsection \ref{subs-orth-pol},  we recall the 
matrix orthogonal polynomials $P$ and $Q$ and we prove a   result 
being a block case generalization (Proposition \ref{defJL}) of a result from \cite{Jitomirskaya1999}. 
It shows a  possibility of a kind of ``a continuous  control'' of  the size    of Jitomirskaya--Last type semi-norms for both matrix orthogonal polynomials $P$ and $Q$ also for the block case.
 Next, in subsection \ref{subs-trans} we study transfer matrices for BJM. We show here how  to use transfer matrices to get a simple alternative proof of the classical algebraic results concerning $P$ and $Q$, called  Liouville--Ostrogradsky formulae (see, e.g., \cite[Theorem 5.2]{Berg2008}). 

  Section~\ref{sec:4} is the central section of this paper. We recall and   study here  the matrix Weyl function $\Wej$ and its relation to properties of $M$ and thus, to spectral properties of $J$. 
The main results  of  subsection \ref{l2mW}  are Theorem~\ref{prop:9ogolniejszy} and Theorem~\ref{duzo-malo-wl2}. Both are generalizations of the scalar case results, and both seems to be slightly more general than the results known from the literature: In the first case see \cite[formula (2.4)]{Teschl2000} for scalar Jacobi matrices; in the second case, see \cite[Proposition 1]{Monotonic} for scalar Jacobi matrices, and see \cite[Theorem 11.1]{Weidmann1987} for higher order ordinary differential operators. 
The first one  gives some important formulae and estimates  for the matrix Weyl function and the so-called Weyl  matrix  solution.  The second gives an upper estimate of the dimension of the space of all the $\ell^2$  generalized  eigenvectors, and it works also for the non-self-adjoint case of $J$.
In the next subsections we briefly recall  the fact that  the Cauchy transform  of the spectral matrix measure of BJO is equal to its  matrix Weyl function (see Fact \ref{fact-WC}) and we use this to study the relations of boundary limits of the matrix Weyl function with  the properties of the spectral matrix measure. This leads us to Theorem~\ref{cor:1}, 
being one of the important  spectral results for BJO, analogic to the appropriate  well-known result for $d=1$ (see e.g. \cite[Lemma 3.11]{Teschl2000}). This result with its  rigorous proof   seems to be absent in the spectral literature  so far.

The last Section \ref{sec:Subord} is devoted to the ``simplest form'' of the notion of nonsubordinacy. The main issue here  is Theorem \ref{prop:16}.  It  seems to be quite a strong,  simple and convenient  spectral result having at the same time a short and ``almost elementary''  proof.
It should be mentioned here also, that the general idea of nonsubordinacy and its spectral consequences for $J$ are  developed more deeply 
in our  parallel paper \cite{BarrierNonsubordinacy}, where we consider a more sophisticated  notion  of   barrier nonsubordinacy.

In Appendix~\ref{vecmacmeas} we collect  basic notions and simple facts concerning general vector measures and matrix measures, used in the previous sections.

\subsection*{Acknowledgment}
The article was partially supported by grant ``Subordinacy for block Jacobi operators. Spectral theory for self-adjoint finitely-cyclic
operators and the introduction to $L^2$ type matrix measure spaces'', NI 3B  POB III  IDUB (01/IDUB/2019/94),   funded by University of Warsaw, Poland.
The  author Grzegorz Świderski was partially supported by long term structural funding -- Methusalem grant of the Flemish Government. Part of this work was done while he was a postdoctoral fellow at KU Leuven.
The author  Marcin Moszyński wishes to thank:
\begin{itemize}
\item[$\cdot$]
Anna Moszyńska  (IPEVP \& MusInvEv, Warsaw)  -- his wife --
for  extraordinary patience and for valuable linguistic help,
 \item[$\cdot$]
Nadia V. Zaleska (EIMI, St. Petersburg \& MusInvEv,  Warsaw) -- his friend --
 for some wise  hints and for invaluable moral support.
\end{itemize}

\section{Preliminaries} \label{sec:2}

Here we collect and fix  some  general notation and terminology for the paper, and we also introduce  several convenient tools, which will be important   in the main  sections.

\subsection{Abbreviations and symbols} \label{abb+symb}

We use  here also  some  ``more or less common'' {  abbreviations} and { symbols} like:

\begin{itemize}
\item[iff:] \hspace{2 em} \ \ {\em if and only if}  
\item[TFCAE:] \hspace{2 em}   \ \   {\em the following conditions are (mutually) equivalent}
\item[w.r.t.:] \hspace{2 em}   \ \ {\em with respect to} 
\item[s.a.:] \hspace{2 em}   \ \   {\em self-adjoint} (for operators)
\item [a.c.:] \hspace{2 em}    \ \   {\em absolutely continuous} (for operators, measures etc.)
\item [sing.:] \hspace{2 em}    \ \   {\em singular} (as above)
\item [a.e.:] \hspace{2 em}    \ \   {\em almost everywhere}
\item [JM, BJM:] \hspace{2 em}    \ \   {\em Jacobi matrix, block Jacobi matrix}, respectively 
\item [JO, BJO:] \hspace{2 em}    \ \   {\em Jacobi operator, block Jacobi operator}, respectively 
\item   [$\lin Y$:] \hspace{2 em}    \ \
 {\em the linear subspace generated by a subset $Y$ of a linear space}
\item   [$\restr{F}{Y}$:] \hspace{2 em}    \ \
 {\em the restriction of   function $F$ to the subset $Y$ of the domain}
 \item   [$F(Y)$:] \hspace{2 em}    \ \
 {\em the image of  subset $Y$ with respect to function $F$}  
 \item   [$\Dom(A)$:] \hspace{2 em}    \ \
 {\em the domain of linear operator $A$} 
 \item   [$\Dom(A^{\infty})$:] \hspace{2 em}    \ \
 {\em the intersection of all   \    $\Dom(A^n)$ for $n\in\NN$.}
\end{itemize}

\subsection{Introductory notation and notions} \label{intr-not}

We   use here the following symbols for some sets of scalars: 
\[
    \CC_+:= \{ z \in \CC : \Im(z) > 0 \},\hspace{1.5em} 
    \RR_+ := \{ t \in \RR: \ t > 0 \},\hspace{1.5em}  
    \N_{k} := \{ n \in \ZZ: n\geq k \} \ \ \ \mbox{ for\ } k\in\ZZ,
\]
so, e.g.,  
\[
    \N = \Nk{1}, \quad  
    \Nz = \N\cup\{0\}, \quad 
    \Nmj = \N\cup\{-1,0\}.
\]

Let us  fix here some  $d\in\N$.  The vectors of the standard base in $\CC^d$ are denoted by  \ 
$
e_1, \ldots,e_d
$. \ 
By $M_d(\CC)$ we denote the space of all  $d \times d$ complex matrices,  with the usual matrix/operator norm. We  identify any $A\in M_d(\CC)$ with the  appropriate  linear transformation of $\CC^d$ induced by matrix $A$. 
In particular, we shall use alternatively both the operator and the matrix notation for the action of  $A$ on vectors from $\CC^d$, namely, for $v\in\CC^d$   we typically  use $Av$, but   sometimes  we also write $Av^{\mathrm{T}}$. 
For $i,j\in \{1, \ldots, d\}$
the term of $A$ from its  $i$th row and $j$th column is denoted as usual by $A_{i,j}$ and $v_j$ is the $j$th term of $v$. The symbol $\col{A}{j}$ denotes  the   $j$th column of $A$ (usually we treat  columns  as  $\CC^d$-vectors and not as one-column matrices). Similarly for    $\row{A}{i}$ ---  the  $i$th row of $A$. Moreover, for vectors $v^{(1)},\dots,v^{(d)}\in\CC^d$  the matrix $A$ with  $\col{A}{j}=v^{(j)}$\ \ for any $j$ is denoted   by $[v^{(1)},\dots,v^{(d)}]$.

If $X$ is a linear space, then by \  $\lnk{k}{X}$ \ we denote the linear space of {\bf all} the sequences $x=(x_n)_{n\in\Nk{k}}$ with terms in $X$ and \  
$\lfinnk{k}{X}$ \ is the subspace of $\lnk{k}{X}$ consisting of all the ``finite'' sequences, i.e., of such  $x$ that $x_n=0$ for $n$ sufficiently large.

For sequences of vectors: \  if $u^{(1)},\dots,u^{(d)}\in\lnk{k}{\CC^d}$ with $u^{(j)}=(u^{(j)}_n)_{n\in\Nk{k}}$ for $j=1,\ldots,d$, then the symbol $[u^{(1)},\dots,u^{(d)}]$, used already above for vectors and not for sequences of vectors, denotes the matrix sequence $U\in\lnk{k}{\Mdc}$ with $U_n:=[u^{(1)}_n,\dots,u^{(d)}_n]$ for any $n\geq k$.

The symbols $\norm{\cdot}_X$ \ and  \  $\sprod{\cdot}{\cdot}_{X}$ denote here the norm in  a normed space $X$ and the scalar product in a Hilbert space $X$, respectively, but we sometimes simplify the subscript $X$ for  $X$-s having some longer symbols and  we  generally   omit it  for all operator norms,  which we use by default  for the bounded operators (mainly for matrices from $\Mdc$),  if no other choice is made.

For  a linear  operator $A : X \strz X$ in a normed space $X \neq \{0\}$  we define its \emph{minimum modulus} by
\begin{equation}
	\label{eq:23}
	\lnorm{A} := \inf_{\| x \|_X = 1} \| A x \|_X.
\end{equation}
Obviously, if $\dim X=1$, then $\lnorm{A}=\norm{A}$.
Recall that 
if  $A$ is invertible, then
\begin{equation} \label{eq:23a}
	\lnorm{A} = \frac{1}{\|A^{-1}\|},
\end{equation}
and  for $0 < \dim X < +\infty$ \  $A$ is invertible \ iff \  $\lnorm{A} > 0$.

We use the symbols 
\[
    \cl{G}, \quad  \clleb{G}
\]
for the ``usual'' (topological) closure of a subset $G$ of a topological space and for the essential closure of a Borel set $G\subset\RR$, respectively  \  (and \ $\sprz{{\la}}$ 
\  is used here  for the complex conjugation of $\la\in\CC$). Recall, that   
\beql{clos-leb}
\clleb{G}:=  
	\{ 
		t \in \RR : \forall_{\varepsilon > 0} \
		|G \cap (t-\varepsilon; t+ \varepsilon)| > 0
	\} 
\eeq
for $G \in \Bor(\RR)$ and 
$|\hspace{0.01em}\cdot\hspace{0.01em}|$ is here  the standard 1-dimensional Lebesgue measure on $\Bor(\RR)$ but, as usual, it will denote also  the absolute value. 

If $\mu$ is a measure\footnote{\label{foot-meas}Here the name ``measure'', without any  extra  adjectives / names of the type vector, matrix, complex, real, spectral etc., denotes always a classical  measure with values in $[0,+\infty]$,  without necessity of adding ``non-negative''. And  the remaining ones,  all  ``adjective (of the  above type) + measures'',  belong to a wide class of  vector measures  for an appropriate  vector space (real or complex). - See Appendix~\ref{vecmacmeas}.}  
on $\frakM$ -- a $\sigma$-algebra of subsets of some $\Omega$, 
and $p\in[1,+\infty)$, then 
 $L^p(\mu)$ (without the ``universum'' $\Omega$ and the $\si$-algebra for the measure, for short) denotes the standard $L^p$ Banach space of the classes of the appropriate complex functions on the ``universum'' for the measure $\mu$.    We need sometimes to distinguish here  $L^p(\mu)$ from the appropriate space of functions (and not classes) denoted here by $\calL^p(\mu)$.  For $p=1$ and $X:=\RR^d$ we also use   $\calL^1_{X}(\mu)$ \label{L1X}
 to denote  the space of integrable  functions  from $\Omega$ into $X$ in the standard coordinatewise 
  sense of the integral and the integrability.

We need also some  basic terminology  concerning vector and matrix measures important for this paper. It is quite long, so it is  collected in Appendix~\ref{vecmacmeas}, instead  of here.  
   Moreover,  for a matrix measure $M$ by  $L^2(M)$ we denote the appropriate $L^2$-Hilbert space induced by this matrix measure (see Section~\ref{L2M} and,  e.g., \cite{Moszynski2022, Duran1997} for more details).

If $X$ is a normed space, then 

\[
    \ldnz{X}:= 
    \bigg\{ x \in \lnz{X} : \sum_{n=0}^{+\infty} \| x_n \|_{X}^2 < \infty \bigg\},
\]
 is a normed space with the norm defined for $x\in\ldnz{X}$  by 
\[
    \norm{x}_{\el^2} := \sqrt{\sum_{n=0}^{+\infty} \| x_n \|_{X\ .}^2}
\]
If, moreover, $X$ is a Hilbert space, then $\ldnz{X}$ is a Hilbert space  with the scalar product given for $x, y\in\ldnz{X}$ by
\beql{l2scal}
	\langle x, y \rangle_{\el^2} := \sum_{n=0}^{+\infty} \langle x_n, y_n \rangle_{X\ .}
\eeq

Here, the most important case for us is the Hilbert space  \ $\ldnz{\CC^d}$. 
As {\em the standard orthonormal basis of} $\ldnz{\CC^d}$ we consider
\[
    \big\{ \delta_n(e_i) \big\}_{(i,n) \in \{1,\ldots,d\} \times \Nz \ ,}
\]
where  for any vector $v \in \CC^d$ and $n\in\Nz$ we define the sequence $\delta_n(v) \in \lfinnk{0}{\CC^d}$ by
\beql{den}
    (\delta_n(v))_m :=
    \begin{cases}
	v & \text{if } m=n \\
	0 & \text{otherwise,}
    \end{cases} \qquad \qquad m\in\Nz.
\eeq
Moreover we have
\beql{lfin-denlin}
    \lfinnk{0}{\CC^d} = \lin\{\delta_n(v): v \in \CC^d, n\in\Nz \}, \ \ \ \mbox{and} \ \ \ \cl{\lfinnk{0}{\CC^d}}=\ldnz{\CC^d}.
\eeq

If  $\calH$ is  a Hilbert space and $A$ --- a   self-adjoint  operator (possibly unbounded) in $\calH$, then  the projection-valued spectral measure (``the resolution of identity'') of (for) $A$  is  denoted by $E_A$. In particular $E_A : \Bor(\RR) \strz \calB(\calH)$, where $\Bor(\RR)$ is the Borel  $\si$-algebra of $\RR$ and $\calB(\calH)$ denotes, as usual, the space of bounded operators on $\calH$. 
If $x,y \in \calH$,  then $E_{A,x,y}$ denotes {\em the spectral measure for $A$, $x$ and $y$}, i.e.  the complex measure  given by
\beql{eq:I:spmes}
    E_{A,x,y}(\om) := \lesc E_A(\om) x, y \risc_{\calH}  , \quad \om\in  \Bor(\RR),
\eeq
and $E_{A,x}:=E_{A,x,x}$   ({\em the spectral measure for $A$ and $x$}).
Denote also
\[
	\calH_{\mathrm{ac}}(A) :=
	\{ x \in \calH : E_{A,x} \text{ is a.c. with respect to the Lebesgue measure on } \Bor(\RR) \}.
\]

If  and $G \in \Bor(\RR)$, then  the symbol $(A)_G$ denotes the part of $A$ in the (invariant reducing) subspace 
$$\calH_G(A) := \Ran E_A(G).
$$ 

Recall one of the important spectral  notions: 

\begin{definition}\label{abs-co}
\emph{$A$ is absolutely continuous\footnote{Note that we  fix it because, unfortunately, there is no common unique terminology for the property defined here.  Several various names are also popular in spectral theory 
for the same or for somewhat ''similar'' property. E.g., {\em  $A$ is purely absolutely  continuous in \dots} or {\em  $A$ has  purely absolutely  continuous spectrum in \dots}. The version used here (for the case $G=\RR$) has a long tradition in spectral theory  and comes from Putnam --- see, e.g., \cite{Putnam1956}.} in $G$} \ \ iff \ $\calH_G(A) \subset \calH_{\mathrm{ac}}(A)$.

\emph{$A$ is absolutely continuous} \ \  iff \  $A$ is absolutely continuous in $\RR$.
\end{definition}

Some further notation and terminology will be  successively introduced   in  next subsections.

\subsection{Some matrix formulae and  inequalities}

Let us recall some notions related in particular to $d\times d$ matrices.  For $A \in \Mdc$:
\begin{itemize}
\item
its \emph{Hilbert--Schmidt norm} is defined by
\begin{equation} \label{eq:65}
    \HSnorm{A}:= \bigg( \sum_{i=1}^d \| A e_i \|^2_{\CC^d} \bigg)^{1/2}.
\end{equation}
\item  its \emph{real and imaginary parts} $\Re(A)$ and  $\Im(A)$ (in the adjoint, and not in the complex conjugation  sense)  are given by 
\begin{equation} \label{eq:150}
    \Re(A) : = \tfrac{1}{2} (A + A^*), \quad 
    \Im(A): = \tfrac{1}{2i} (A - A^*).
\end{equation}
\end{itemize}
For $v\in\CC^d$ we shall use the symbol $E^v$ to  denote the matrix / operator $[v,0,\ldots, 0]\in\Mdc$, i.e.,  
\beql{def:Ev}
    E^v(e_j) = \opcje{v}{j=1}{0}{j>1,} \quad j=1,\ldots, d.
\eeq

\begin{proposition} \label{prop:18}
Let  $A \in \Mdc$ and $v \in \CC^d$. Then:
\begin{enumerate}[label=(\roman*)]
    \item \label{prop:18:1}
     $\norm{A} \leq \HSnorm{A}$.
    \item \label{prop:18:2}
    $\norm{\Im(A)} \leq \norm{A}, \ \  \norm{\Re(A)} \leq \norm{A}$.
    \item \label{prop:18:reimsc}
    $\Re\lesc \nu, A \nu \risc_{\CC^d}=\lesc \nu, (\Re A) \nu \risc_{\CC^d}, \quad\Im\lesc \nu, A \nu \risc_{\CC^d}=-\lesc \nu, (\Im A) \nu \risc_{\CC^d}$.
    \item \label{prop:18:3}
    $\norm{AE^v} = \norm{Av}_{\CC^d}$.
    \item \label{prop:18:4}
    If $A \geq 0$, then \ \ \ $	\norm{A} \leq \tr A \leq d \norm{A}$.
\end{enumerate}
\end{proposition}

\begin{proof}
Part~\ref{prop:18:1} is a classical result (easy to obtain by the Schwarz inequality), and  \ref{prop:18:2} follows from $\norm{A^*}=\norm{A}$ and the triangle inequality.
Part \ref{prop:18:reimsc} follows  directly from the definitions of $\Re, \Im$ for complex numbers and for matrices --- in \eqref{eq:150}.

To get~\ref{prop:18:3} observe first that $AE^v=E^{Av}$ by \eqref{def:Ev}, so it suffices to consider $A=I$. But using~\ref{prop:18:1} we get $\norm{E^v} \leq \norm{v}_{\CC^d}$ and $\norm{E^v} \geq \norm{E^ve_1}_{\CC^d}=\norm{v}_{\CC^d}$, so the equality holds.

Suppose that  $A \geq 0$ and let $\la_1\ldots, \la_d$ be all  the eigenvalues of $A$ repeated according to  their multiplicities. 
We thus get~\ref{prop:18:4} by
\[
    \norm{A} = 
    \max_{1 \leq i \leq d} \la_i \leq
    \sum_{i=1}^d \la_i = 
    \tr A \leq 
    d\max_{1 \leq i \leq d}\la_i =
    d \norm{A} \qedhere.
\]
\end{proof}
 
\subsection{Asymptotic  symbols and the affine interpolation}

Let   $S$  be an arbitrary set and   
$f, g:S\strz \CC$. We define the symbol $\asymp$ of ``{\em asymptotic similarity}'' of functions:

\beql{gae}
 f\asymp  g \ \  \iff \ \ \Exists_{c,C\in\RR_+} \Forall_{s\in S}\ \ c|g(s)|\leq |f(s)|\leq C |g(s)|
\eeq
(note the presence of the  absolute value in this definition). We shall use also alternative  notation: $f(s) \asympp{s} g(s)$.

If $S = \Nk{k}$ for some $k \in \ZZ$,  then $f$ is just a sequence from
$\lnk{k}{\CC}$, i.e. $f=(f(n))_{n \in \Nk{k}}=(f_n)_{n \in \Nk{k}}$, and  we shall consider its {\em affine interpolation} \  
$\af{f}:[k,+\infty)\strz\CC$, uniquely defined by the conditions:
\begin{enumerate}[label=(\alph*)]
    \item $\restr{\af{f}}{\Nk{k}}=f$,
    \item for each $n\in\Nk{k}$ \ \  $\restr{\af{f}}{[n, n+1]}$ \ is an affine function, i.e. a  function of the form \\ $[n, n+1]\ni t\mapsto at+c\in\CC$ \  for some    $a, c\in \CC$  (depending here  also on $n$). 
\end{enumerate}

In particular $\af{f}$ is a continuous interpolation of $f$, and one can easily check that it can be also given by the explicit formula:
\beql{af-form}
 \af{f}(t) = f_{\lfloor t \rfloor} + \{t\} \left(f_{\lfloor t \rfloor + 1} - f_{\lfloor t \rfloor}\right),\quad t \in [0, \infty)
\eeq
where $\{t\} = t - \lfloor t \rfloor$ is the fractional part of $t$.

The following  result joining the asymptotic similarity of sequences and of their  affine interpolations, will be 
convenient in the proof of our main result.

Let us prove first  the following   
result on quotients of two affine functions.
\begin{lemma}\label{lem:aff}
 Suppose that    $\alpha,  \beta, a_1, c_1, a_2, c_2 \in \RR$,  $\alpha < \beta$ and $a_2 t + c_2 \neq 0$ for any $t \in [\alpha, \beta]$.   Let  $\varphi : [\alpha, \beta] \strz \RR$ \ be given by 
\[
    \varphi(t) :=  \frac{a_1 t + c_1}{a_2 t + c_2},   \quad t \in [\alpha, \beta].
\]
Then $\varphi$ is monotonic and, in particular, its  maximal and  minimal value   is 
attained on  the set   $\{\alpha, \beta\}$.
\end{lemma}
\begin{proof}
We have:
\[
    \varphi'(t) :=  \frac{a_1c_2-a_2c_1}{(a_2t+c_2)^2}
\]
and thus, we have only two cases:
\begin{itemize}
    \item $a_1 c_2 - a_2 c_1 = 0$, hence $\varphi$ is constant;
    \item $a_1 c_2 - a_2 c_1 \neq 0$, so $\varphi'$ has no zeros.
\end{itemize}
In both cases our assertion holds.
\end{proof}

\subsection{Analogs of ``J--L continuous interpolation''  
of a discrete  family of semi-norms} \label{J-Lsemi}

The main idea of  the Jitomirskaya--Last's new approach in   \cite{Jitomirskaya1999}  was ``technically'' based  on a continuous interpolation
of the discrete  family of semi-norms $\{\norm{\cdot}_{n}\}_{n \in \Nz}$ in $\lnz{\CC}$ to the family  $\{\norm{\cdot}_{t}\}_{t \in [0, +\infty)}$.

 For our purposes we shall extend this notion in two ways. Namely, let $V$ be a normed space,  $n_0 \in \ZZ$ and assume that   $X \in \ell(\Nk{ n_0}, V)$.
 
 For any $(n_1,t) \in \ZZ \times \RR$ such that $n_0 \leq n_1 \leq t$, we define 
\begin{equation} \label{eq:105a}
    \norm{X}_{[n_1,t]}:= 
    \bigg( 
    \sum_{k=n_1}^{\lfloor t \rfloor} 
        \norm{X_k}_V^2 + \{t\} \! \norm{X_{\lfloor t \rfloor + 1}}^2_V 
    \bigg)^{1/2},
\end{equation}
where $\lfloor t \rfloor$ and $\{t\}$ are the integer  and the fractional part of $t$, respectively,  and for $t=\infty$
\begin{equation} \label{eq:105b}
    \norm{X}_{[n_1,\infty]} := 
    \bigg( 
    \sum_{k=n_1}^{+\infty} \norm{X_k}^2_V
    \bigg)^{1/2}
\end{equation}
(which can  possibly be $+\infty$).

Moreover, if $V= \Mdc)$ for some $d \geq 1$, then similarly to \eqref{eq:105a} we define a continuous family based on minimum modulus (see \eqref{eq:23}) instead of matrix norm. So,   for any $(n_1,t) \in \ZZ \times \RR$ such that $n_0 \leq n_1 \leq t$ we consider 
\begin{equation} \label{eq:105c}
    \lnorm{X}_{[n_1,t]} = 
    \bigg( 
    \sum_{k=n_1}^{\lfloor t \rfloor} 
        \lnorm{X_k}^2 + \{t\} \!\! \lnorm{X_{\lfloor t \rfloor + 1}}^2 
    \bigg)^{1/2}.
\end{equation}

By \eqref{af-form} we can relate the above constructions with the notion of affine extension introduced in the previous subsection. The squares of the ``new objects'' are just the affine extensions of their discrete counterparts (see \eqref{disc-semi-a}  and   \eqref{disc-semi-b} below),
which are somewhat more natural and simpler for the context of operators ``acting on sequences'' considered in this paper.

\begin{observation}\label{prop:15}
The   notions defined by \eqref{eq:105a} and \eqref{eq:105c}  (i.e, for $t<+\infty$) 
satisfy respectively 
\[
    \norm{X}^2_{[n_1,t]} = \af{S_{X,n_1}}(t)
\]
and
\[
    \lnorm{X}^2_{[n_1,t]} = \af{\check{S}_{X,n_1}}(t)
\]
for  $t\geq n_1$, \ 
where \  $S_{X,n_1}, \ \ \check{S}_{X,n_1} :\Nk{n_1}\strz\RR$ \  are given for $n\geq n_1$ by
\beql{disc-semi-a}
    S_{X,n_1}(n) := 
    \norm{X}^2_{[n_1,n]} = 
    \sum_{k=n_1}^{n}\norm{X_k}^2,  
\eeq
\beql{disc-semi-b}
    \check{S}_{X,n_1}(n) := 
    \lnorm{X}^2_{[n_1,n]} = 
    \sum_{k=n_1}^{n}\lnorm{X_k}^2_V.
\eeq
\end{observation}

 Note also that in the scalar case $d=1$ for $X \in \ell(\Nk{ n_0}, \Mdc)$ we simply  have $\norm{X}_{[n_1,t]} = \lnorm{X}_{[n_1,t]}$ and  $S_{X,n_1} = \check{S}_{X,n_1}$.

Consider a second sequence $Y \in \ell(\Nk{ n_0}, V)$. 
Using the above Observation and Lemma~\ref{lem:aff} we get:

\begin{corollary} \label{szac-iloraznort-norn}
Suppose that 
$n_1 \in \ZZ$,   $t\in \RR$, $n_0 \leq n_1 \leq t$. \ Let \ 
$n:=\lfloor t \rfloor$ and assume   that 
$\norm{Y}_{[n_1,n]} \neq 0$. \ \ 
Then there exist such $\npod, \nnad \in \{n, \ n+1\}$  that
\[
    \frac{\norm{X}_{[n_1,\npod]}}{\norm{Y}_{[n_1,\npod]}} \leq 
    \frac{\norm{X}_{[n_1,t]}}{\norm{Y}_{[n_1,t]}} \leq
    \frac{\norm{X}_{[n_1,\nnad]}}{\norm{Y}_{[n_1,\nnad]}}.
\]
If $V = \Mdc$ then the analogous result with all the ``$\norm{\cdot}$'' replaced by 
``$\lnorm{\cdot}$'' is also true.
\end{corollary} 
\begin{proof}
It suffices to use Observation~\ref{prop:15} and then  Lemma~\ref{lem:aff} to the function $\varphi$ given on $[n, \ n+1]$  by the formula  \ $\varphi(t) := \left(\frac{\norm{X}_{[n_1,t]}}{\norm{Y}_{[n_1,t]}} \right)^2$.
\end{proof}

The above  result is covenient when we consider two variants of definitions of ``subordination type'' notions: the first one --- based on the ratio of the seminorms only for the `discrete $n\in\N$', and the second --- based on the ratio of the seminorms  for all the `continuous $t>0$'. It often allows to prove the equivalence of such two approaches (see, e.g., subsection \ref{sec:m-v-nonsubodinacy}).

\section{Block Jacobi operator. The spectral matrix measure and a spectral representation} \label{sec:3n}

We denote ``the size of the block'' for BJM by    $d$, $d\in\NN$, and for the whole  paper we assume that  $(A_n)_{n \in \Nz}$ and $(B_n)_{n \in \Nz}$ are  sequences of matrices from $M_d(\CC)$ such that 
\beql{zal-odw+sym}
    \det A_n \neq 0, \ \ \  
    B_n = B_n^*, \quad n \in \Nz.
\eeq
Define a \emph{block Jacobi matrix}  
\begin{equation*}
    \calJ =   
    \begin{pmatrix}
        B_0	& A_0	&		&       \\
        A_0^*	& B_1	& A_1	&         \\
            & A_1^*	& B_2 	& \ddots   \\
            &		&		\ddots &  \ddots 
    \end{pmatrix}.
\end{equation*}
The pair of the sequences $(A_n)_{n \in \NN_0}$ and $(B_n)_{n \in \NN_0}$ will be called \emph{Jacobi parameters of $\calJ$}.
In fact,  we mean that    $\calJ$ is the linear operator ({\em ``the formal block Jacobi operator''}) acting on  the linear space $\ell(\Nz, \CC^d$) of all the $\CC^d$ sequences, and  the   action of $\calJ$ is well-defined via formal matrix multiplication:   
\[
    \calJ : \ell(\Nz, \CC^d)\strz \ell(\Nz, \CC^d), 
\]
i.e.,  for any  $u\in\ell(\Nz, \CC^d)$
\beql{eq:defformal}
    \left(\calJ u\right)_n :=
    \opcje{ B_0 u_0 + A_0 u_{1}}{n=0}{	A_{n-1}^* u_{n-1} + B_n u_n + A_n u_{n+1}}{n\in\N.} 
\eeq

Recall  now two important  operators related to $\calJ$ and acting in the Hilbert space $\ldnz{\CC^d}$: \  $\Jmin$ and \  $J$  (which will  coincide  in most of our further considerations). 
The operator $\Jmin$ ({\em the minimal block Jacobi operator}) is simply the closure 
in $\ldnz{\CC^d}$ 
of the operator in $\ldnz{\CC^d}$ being the restriction: \  $\restr{\calJ}{\lfinnz{\CC^d}}$. 
And to define 
$J$ ({\em the maximal block Jacobi operator}) we first choose its  domain in  a usual way for ``maximal-type'' operators:
\beql{maxJdom}
    \Dom(J) := \big\{ u\in\ldnz{\CC^d}: \calJ u\in\ldnz{\CC^d} \big\},
\eeq
and then we  define \  $J:=\restr{\calJ}{\Dom(J)}$.
   
When both $(A_n)_{n \in \Nz}$ and $(B_n)_{n \in \Nz}$ are bounded sequences, then obviously we have only one operator $\Jmin=J\in\calB(\ldnz{\CC^d})$, which is symmetric, so  self-adjoint. But here we consider also unbounded cases, so let us recall  the following result 
\begin{fact}\label{factsaJ}
$(\Jmin)^*=J$. Moreover, TFCAE:
\begin{enumerate}[label=(\roman*)]
    \item $\Jmin$ is s.a.,
    \item $J$ is s.a.;
\end{enumerate}
and if  one of the above conditions holds, then   $\Jmin=J$.
\end{fact}
See \cite[Chapter VII.§2.5]{Berezanskii1968}  for\footnote{Actually, in \cite[Chapter VII.§2.5]{Berezanskii1968} only the case $A_n = A_n^*$ for all $n \geq 0$ was considered, but the proof in the general case is similar.} $\Jmin^* = J$, and the remaining part follows from this by \cite[formula (7.1.22)]{Simon2015Book4}.

For the main results of the theory presented here the maximality of the domain is crucial. Moreover, we study  mainly ``the self-adjoint case'', so (by Fact~\ref{factsaJ}) the best choice here is to focus   just   only on the operator $J$, later on\footnote{So, it means  in particular that  we do not study here any other possible s.a. extensions of $\Jmin$ than $J$. Note that the same is important for the classical ($d=1$) subordinate theory --- see \cite{Khan1992}.}.

By \eqref{maxJdom} and \eqref{eq:defformal} we get 
\beql{fin}
    \lfinnz{\CC^d}\subset\Dom(J), \qquad 
    J\left(\lfinnz{\CC^d}\right) \subset \lfinnz{\CC^d},
\eeq
so in particular $J$ is densely defined by \eqref{lfin-denlin}. Moreover, by \eqref{eq:defformal} we compute
\beql{Jnadel}
    J\left( \delta_n(v)\right) =
    \delta_{n-1}(A_{n-1}v)+\delta_{n}(B_{n}v)+\delta_{n+1}(A_{n}^*v), \qquad 
    v \in \CC^d,\  n\in\Nz, 
\eeq
where we additionally denote
\[
    \delta_{-1}(v) := 0, \qquad v \in \CC^d.
\]

\subsection{The finite-cyclicity and the canonical cyclic  system  } \label{fincyc}

In the scalar $d=1$ case  Jacobi operator  $J$ is cyclic with a cyclic vector $\varphi := \de_{0}(e_1)$ ({\em the canonical cyclic vector for $J$}), which means that the space 
\beql{cyclic}
    \lin\{J^n\varphi: \ n \in \Nz\} 
\eeq
is dense in $\ldnz{\CC}$ (here, for $d=1$, we simply have $e_1=1\in\CC$).
--- Indeed,  this is known that   the above space is just  equal to $\lfinnz{\CC}$ for such  $\varphi$. Cyclicity plus self-adjointness provides a very simple spectral representation of the operator.

So, suppose now, that ``the scalar'' $J$ is s.a. and  consider: \ 
 $\xx$ --- the identity function on $\RR$, \  $\xx(t)=t$ for $t\in\RR$, and $\mu$ --- ``the scalar'' spectral measure  for $J$ and $\varphi$, \ i.e. \  $\mu = E_{J,\varphi}$ (see, Section~\ref{intr-not} for the spectral notation).

By the  well-known spectral result, $J$,  as  a cyclic s.a. operator,  is unitary equivalent  to the operator of the multiplication by $\xx$ in the space $L^2(\mu)$.
 
 This one-dimensional  result has also its analog for the block-Jacobi case with arbitrary $d$.
 First of all, instead of cyclicity, we consider here the so-called  {\em finite-cyclicity} notion (see \cite{Moszynski2022}). It  means that for some $k\in\NN$  there exists   {\em a cyclic system 
 $\vec{\varphi} = (\varphi_1, \ldots, \varphi_k)$    for $J$}, i.e., a system of such  vectors from $\Dom(J^{\infty})$, that 
 \beql{fincyclic}
\lin\{J^n\varphi_j: \ n\in\Nz, \ j=1,\ldots,k\} 
\eeq
is dense in $\ldnz{\CC^d}$.
In our general $d$-dimensional case the choice of  $\vec{\varphi}$ can be done   in an analogical way, as it was  for $d=1$, namely define
\beql{fix-fincyc}
    \vec{\varphi} := (\varphi_1, \ldots, \varphi_d), \quad
    \varphi_j := \de_{0}(e_j), \ \  j=1,\ldots,d.\ \ \eeq
The following  result is true.

\begin{proposition} \label{prop-fincyc}
The system $\vec{\varphi}$ is a cyclic system for $J$.
\end{proposition}

\begin{proof}
For each  $n \in \NN_0$ denote\footnote{Here we use the standard operation of the sum of subsets of a linear space, i.e., $\sum_{k=0}^n D_k:= \{\sum_{k=0}^nx_k:\ x_k\in D_k \ \mbox{ for any } k=0, \ldots n \}$ \ for subsets $D_0, \ldots D_n$.}
\beql{for:XnYn}
    X_n := \left\{ x \in \ldnz{\CC^d}: \Forall_{m \neq n} \  x_m = 0 \right\}, \hspace{3em} 
    Y_n := \sum_{k=0}^n X_k.
\eeq
So   \ $X_n, Y_n\subset\lfinnz{\CC^d}$ \  and
\beql{for:Xnde}
    X_n = 
    \left\{ \de_{n}(w): w \in \CC^d \right\} = 
    \lin \left\{ \de_{n}(e_j): \ \  j=1,\ldots,d \right \}, \qquad n \in \Nz.
\eeq
In particular 
\beql{X0phi}
    X_0 = \lin \left\{ \varphi_j: \ \  j=1,\ldots,d \right\}.
\eeq

Using \eqref{for:Xnde} and \eqref{Jnadel} we get
\[
    J(X_n) \subset Y_{n+1}, \qquad n \in \Nz,
\]
so  also 
\beql{JYn}
    J(Y_n) \subset Y_{n+1}, \qquad n \in \Nz.
\eeq
Using this by the obvious induction  we obtain
\beql{JnX0}
    J^n(X_0) \subset Y_{n}, \qquad n \in \Nz.
\eeq

Denote 
 $Y_{-1}:=\{0\}$ and
for $k \in \NN_0$ denote also
\beql{forCk}
    C_k := \opcje{(A_0\cdot \cdots \cdot A_{k-1})^*}{k>0}{I}{k=0,}
\eeq
so in particular $C_k$ is invertible.
To continue the proof, we shall first prove the following  two results.

\begin{lemma}\label{lem-fincyc}
For any $k\in\NN_0$
\beql{n-lem-fincyc}
    \Forall_{w \in \CC^d} \  \Exists_{u\in Y_{k-1}} \  J^k\de_{0}(w)=u+\de_{k}(C_kw).
\eeq
\end{lemma}

\begin{proofof}{Lemma~\ref{lem-fincyc}}
For $k=0$ and $w\in\CC^d$ we have 
$J^k \de_{0}(w) = \de_{0}(w) = 0 + \de_{0}(C_0w)$, so \eqref{n-lem-fincyc} holds for $k=0$.
Suppose now that it holds for some $k \in \Nz$. Then by \eqref{n-lem-fincyc} for $k$ and by \eqref{Jnadel}, for $w\in\CC^d$
\[
    J^{k+1} \de_{0}(w) = 
    J(J^k \de_{0}(w)) = 
    J(u) + J \left( \de_{k}(C_kw) \right) = 
    J(u) + u' + \de_{k+1}(A_k^*C_kw),
\]
with some $u \in Y_{k-1}$, $u' \in Y_{k}$. Finally, by \eqref{JYn} and \eqref{forCk}
\[
    J^{k+1} \de_{0}(w) = 
    u'' + \de_{k+1}(A_k^*C_kw) = 
    u'' + \de_{k+1}(C_{k+1}w),
\]
with some $u''\in Y_{k}$. So \eqref{n-lem-fincyc} holds for $k+1$, and we obtain  our assertion by  induction.
\end{proofof}

We shall use this lemma to prove the result below.

\begin{fact}\label{fact-fincyc}
For any $n \in \NN_0$
\beql{n-fact-fincyc}
    Y_n = \sum_{k=0}^n J^k(X_0).
\eeq
\end{fact}

\begin{proofof}{Fact~\ref{fact-fincyc}}
We get ``$\supset$'' from \eqref{JnX0}. Let us prove ``$\subset$'' by induction. For $n=0$ the assertion is obvious. Now  consider $n\in\Nz$ and suppose that 
$
    Y_n \subset \sum_{k=0}^n J^k(X_0)
$.

Then 
\[
    \lin \left( Y_n \cup \ J^{n+1}(X_0) \right) \subset 
    \sum_{k=0}^{n+1} J^k(X_0),
\]
by the linearity of the RHS. Thus, by \eqref{for:XnYn},  to get 
$
    Y_{n+1} \subset \sum_{k=0}^{n+1} J^k(X_0)
$, it suffices to prove
\beql{Xnpj}
    X_{n+1} \subset \lin \left( Y_n\cup\ J^{n+1}(X_0) \right).
\eeq
Indeed, if $w\in\CC^d$, then by Lemma~\ref{lem-fincyc} 
\[
    \de_{n+1}(w) = v + J^{n+1} \de_{0}(C_{n+1}^{-1}w),
\]
with some $v \in Y_n$. Therefore, by \eqref{for:Xnde} we get 
\eqref{Xnpj}.
\end{proofof}

Let us continue the proof of Proposition~\ref{prop-fincyc}.
Using again  the standard argumentation concerning the linear spaces generated by a subset,
we see by \eqref{X0phi} that 
\beql{ozn}
    \lin \{ J^k \varphi_j:  k=0,\ldots,n, \ j=1,\ldots,d \} = 
    \sum_{k=0}^n J^k(X_0), \qquad n \in \Nz.
\eeq
So, by Fact~\ref{fact-fincyc} and \eqref{ozn}
\[
    \lin \{ J^n \varphi_j: \ n \in \Nz, \ j=1,\ldots,d \} =
    \bigcup_{n=0}^{+\infty} \lin \{ J^k\varphi_j:  k=0,\ldots,n, \ j=1,\ldots,d \} = 
    \bigcup_{n=0}^{+\infty} Y_n = 
    \lfinnz{\CC^d}, 
\]
and the density of $\lfinnz{\CC^d}$ in $\ldnz{\CC^d}$ finishes the proof.
\end{proof}

Surely, this choice of a cyclic system for $J$ is not unique. The particular  $\vec{\varphi}$ defined by \eqref{fix-fincyc} is called {\em the canonical cyclic  system  for $J$}.
 
\subsection{The spectral matrix measure $M$  
and the representation of $J$ as the multiplication operator.} \label{L2M}
 
The next notions  which should be generalized, when we are moving from the scalar to the block case, is the classical $L^2$- measure type  Hilbert space with   ``the scalar  non-negative''   spectral measure $\mu$ for JO  $J$, when $J$ is s.a.  The spectral measure $\mu$ is  defined on Borel subsets of $\RR$, and it   should be replaced by a  less-known object --- the so-called {\em spectral matrix measure} $M$ for BJO $J$, if we assume that it is s.a. And so, the classical   $L^2(\mu)$ Hilbert space will  be replaced by a new  type Hilbert space --- the  $L^2$- matrix measure space. The spectral matrix measure for BJO $J$ is a particular example of the general notion of matrix measure (see Appendix~\ref{vecmacmeas}).
 Similarly to the ``scalar'' spectral measure for  the  JO case, it is  defined on  $\Bor(\RR)$, but the values of $M$ are non-negative  $d\times d$ matrices, instead of  non-negative numbers.  Is tightly related to $J$, as in the JO case.  E.g., $J$ can be recovered from its spectral matrix measure  up to a unitary equivalence, as follows from the representation Theorem~\ref{unit-rep-bl} below.

 Let's assume for the rest of this subsection  that  $J$ is s.a.  We use here the following terminology. Let $\vec{\varphi}=(\varphi_1, \ldots, \varphi_k)$ be the canonical cyclic system for $J$ given by \eqref{fix-fincyc}.

\begin{definition}\label{smm}
{\em The  spectral matrix measure $M=M_J= E_{J,\vec{\varphi}}$  \footnote{See also \cite{Moszynski2022}; note the more general notation $E_{A,\vec{\psi}}$ used there for  the spectral matrix measure of any  s.a. finitely cyclic  operator $A$ and  a cyclic system $\vec{\psi}$ for $A$.}  \  for    $J$} is given by  
\begin{equation}
    \label{eq:IV:2.1}
    M: \Bor(\RR) \to \Mdc, \qquad 
    M(\om):= 
    \Big(E_{J,\varphi_j, \varphi_i}(\om)  \Big)_{i,j=1,\ldots,d}\in \Mdc  , \quad 
    \om \in \Bor(\RR)
\end{equation}
(recall that $E_{J,\varphi_j, \varphi_i}(\om)= \lesc E_J(\om)\varphi_j, \varphi_i \risc_{\el^2}$ by \eqref{eq:I:spmes}).
 \end{definition}\ \

We use rather $M$ and not   $M_J$ in this paper to denote the  spectral matrix measure   for    $J$,  when $J$ is fixed.

The analog of the scalar-Jacobi unitary representation result, mentioned before, can be formulated for  our block-Jacobi case in a short way,  as follows.

\begin{theorem}\label{unit-rep-bl}
 $J$ is unitary equivalent  to the operator of the multiplication by $\xx$ in the space $L^2(M)$.
\end{theorem}
For the full formulation and  the  proof (in the general finitely cyclic case) see  \cite[$\xx$MUE Theorem] {Moszynski2022}\footnote{Also the detailed definition of the multiplication by a function operator in $L^2$-matrix measure spaces is presented there.}. 

In particular, the above  result means that  the  spectral matrix measure $M$ of   $J$ ``contains''  all the important spectral information about $J$, similarly to the  spectral measure $\mu$ in the scalar Jacobi case.
Thus properties of $M$ are  important for spectral studies in many concrete examples. For  instance, for such typically studied spectral information, as: the   absolute continuity (or   the singularity) of $J$ in  some  subset of $\RR$, the location of the  spectrum and of some  particular parts of spectra  of $J$, etc.

On the other hand, ``the nice properties'' of the trace measure   $\tr_M$ (see \ref{trM}) for   matrix measure $M$  suggest that  instead of dealing with spectral matrix measure, somewhat sophisticated at times,      
it could be  more useful to deal with its trace measure, being just a classical object --- a finite Borel measure. We can see it, e.g.,  when we try to  ``control'' the above mentioned spectral properties of $J$ related,  to the   absolute continuity or singularity.   In   Fact~\ref{prop:III:3}, Fact~\ref{acscM}, and  Lemma~\ref{techLem} from  Appendix~\ref{vecmacmeas} we  discussed  this problem in more details for its  abstract (vector) measure theory aspect. Their  main ``spectral operator theory'' consequences for $J$ are  Proposition~\ref{cor:1} and Proposition~\ref{thm:1} in Section~\ref{sec:4}. 

At the end of this subsection,  to better illustrate  the notion of the spectral matrix measure for $d>1$,  we compute $M$  in a  general case of   ``simplest'' block Jacobi operators, namely in  the case of diagonal  parameters of  the block Jacobi matrix.

\begin{example}\label{smm-diag}

Consider $d\geq1$ and  diagonal blocks defining   $d\times d$ block Jacobi operator $J$ (the maximal one, as usual here):
\[
    A_n =
  \begin{pmatrix}
        a_n^{(1)} &  & \\
                   & \ddots& \\
        &  & a_n^{(d)}
    \end{pmatrix},  
    \quad
    B_n =
    \begin{pmatrix}
        b_n^{(1)} &  & \\
                   & \ddots& \\
        &  & b_n^{(d)}
    \end{pmatrix},
\]
where $a^{(i)} = \left( a_n^{(i)} \right)_{n \in \NN_0}, \ \  b^{(i)} = \left( b_n^{(i)} \right)_{n \in \NN_0}$ are Jacobi parameters of  $d$ scalar Jacobi matrices $\calJ^{(i)}$ for $i=1,\ldots d$. 
Let $J^{(i)}$ be the appropriate  (maximal) Jacobi operators, and assume that they are all s.a.. Define the 
unitary transformation $U: \ldnz{\CC^d} \strz \bigoplus_{j=1}^d\ldnz{\CC}$ \ \  by
\begin{equation} \label{eq:134-d}
    (Uu)_i:=
    \left((u_n)_i\right)_{n\in\Nz}, \qquad  i=1,\ldots d,
\end{equation}
for \  $u=\left(u_n\right)_{n\in\Nz}\in\ldnz{\CC^d}$.  

Using the maximality of all  the $d+1$ operators $J$, $J^{(1)}, \ldots J^{(d)}$ one can easily check that
\begin{equation} \label{eq:135'}
    J = U^{-1}\left(\bigoplus_{j=1}^d J^{(s)} \right)U. 
\end{equation}

Hence, in particular, also $J$ is s.a.  Moreover,  by the general properties of projection-valued spectral measures  for s.a. operators (for direct sums of operators and for the unitary transfer of operator, 
in both cases by the uniqueness of the projection-valued spectral measure) we have
\begin{equation} \label{E-J-d}
E_J(\om) = U^{-1}\left(\bigoplus_{j=1}^dE_{J^{(j)}}(\om) \right) U, \quad \om\in\Bor(\RR).
\end{equation}
The above formulae allow us to express the spectral matrix measure $M$
 of our block Jacobi operator $J$ 
directly by the $d$ ''scalar'' spectral measures  $\mu^{(i)}$ for  $J^{(i)}$, $i=1,\ldots d$. Recall that 
\begin{equation} \label{muj}
 \mu^{(i)}(\om) :=  \left\langle E_{J^{(i)}}(\om) \delta_0(1), \delta_0(1)
\right\rangle_{\ldnz{\CC}} , \quad \om\in\Bor(\RR), \ \  i=1,\ldots d
\end{equation}
and 
that $\vec{\varphi}$ used in \eqref{eq:IV:2.1} is the canonical cyclic system for $J$ in $\ldnz{\CC^d}$ given by \ $\varphi_i:=\delta_0(e_i)$ \ (see subsection \ref{fincyc}).  
Therefore,  by the unitarity  of $U$ and  by \eqref{E-J-d}  
\begin{align*}
    (M(\om))_{ij} 
    &=
    \left\langle 
    E_J(\om) \delta_0(e_j),
    \ \delta_0(e_i) 
    \right\rangle_{\ldnz{\CC^d}} \\
    &=
    \left\langle 
\left(\bigoplus_{s=1}^dE_{J^{(s)}}(\om)\right) 
 U\delta_0(e_j),
\  U \delta_0(e_i) 
\right\rangle_{
\bigoplus_{s=1}^d\ldnz{\CC}
}\\
&= \sum_{s=1}^d
    \left\langle  
    E_{J^{(s)}}(\om)
 \left(U\delta_0(e_j)\right)_s,
\  \left(U \delta_0(e_i)\right)_s
\right\rangle_{\ldnz{\CC}}
\end{align*}
 for any $\om\in\Bor(\RR)$ and $i,j=1,\ldots d$. But by \eqref{eq:134-d} 
$$
\left(U\delta_0(e_j)\right)_s=\opcje{\delta_0(1)}{s=j}{0}{s\not=j\,,}\qquad s=1,\ldots d,
$$
so
$$
 (M(\om))_{ij} 
    =\opcje{
    \left\langle 
    E_{J^{(i)}}(\om) \delta_0(1),
    \ \delta_0(1) 
    \right\rangle_{\ldnz{\CC}}
    }{i=j}{0}{i\not=j,}
$$
 and finally, by \eqref{muj}, we get the diagonal form of our  spectral matrix measure
\begin{equation} \label{eq:142'}
    M(\om) =
\begin{pmatrix}
        \mu^{(1)}(\om) &  & \\
                   & \ddots& \\
        &  & \mu^{(d)}(\om)
    \end{pmatrix},\quad \om\in\Bor(\RR).
\end{equation}

\end{example}

\section{The associated difference equations and transfer matrices} \label{sec:3s}

\subsection{The two associated difference equations and ``the solution extensions to -1''}
\label{sec:twoeq}

For any $z \in \CC$ let us consider  two 
difference equations tightly related to $\calJ$. The first --- ``the vector'' one --- is the infinite system of equations for a sequence $u=(u_n)_{n\in\Nz}\in\lnz{\CC^d}$:

\begin{equation} \label{eq:2}
    (\calJ u)_n = z u_n, \qquad n \geq 1.
\end{equation}
Each such a vector sequence $u$ is  called {\em generalized eigenvector (for $J$ and $z$)}\footnote{Note here, that it is ``generalized'' for two reasons; the first, because $u$ may not belong to $\Dom(J)$,  not even $\ldnz{\CC^d}$, and the second, since we do not require the equality  for $n=0$ above.} ---  {\em ``gev''} for short. By \eqref{eq:defformal} equivalently its  explicit form can be  written
\beql{eq:2'}
    A_{n-1}^* u_{n-1} + B_n u_n + A_n u_{n+1} = z u_n, \qquad n \geq 1.
\eeq
The second difference equation --- ``the matricial'' one --- is the analog equation (with the right-side multiplication choice\footnote{The left-side one is also possible and used for several reasons, but we shall not consider it here.})  for a matrix sequence  
$U = (U_n)_{n \in \Nz} \in \lnz{\Mdc}$:
\begin{equation} \label{eq:1}
    A_{n-1}^* U_{n-1} + B_n U_n + A_n U_{n+1} = z U_n, \qquad n \geq 1,
\end{equation}
and each such a matrix sequence $U$ is  called {\em matrix generalized eigenvector (for $J$ and $z$)} ---  {\em ``mgev''} for short.

Having our BJM $\calJ$ fixed, for $z\in\CC$ we denote
\[
    \GEVze{z} := \{u\in\lnz{\CC^d}:\, \mbox{$u$ is  a gev for $J$ and $z$}\}
\]
and parallelly
\[
    \MGEVze{z} := \{ U \in \lnz{\Mdc}:\, \mbox{$U$ is  a  mgev for $J$ and $z$} \},
\]
being obviously liner subspaces of  $\lnz{\CC^d}$ and $\lnz{\Mdc}$, respectively. 
By \eqref{zal-odw+sym}, both recurrence relations are of  degree 2, in the sense that for any {initial condition} $\big( C_0, C_1 \big)\in(\Mdc)^2$ there is a unique sequence $U$ satisfying~\eqref{eq:1} with $U_0=C_0$, $U_1=C_1$, and analogously for \eqref{eq:2'}. More precisely, one can easily check the following.
\begin{fact}\label{fact:ini-iso}
For any   $z \in \CC$ the map \ \ 
$\Ini_{z;0,1}: \GEVze{z} \strz \left( \CC^d \right)^2$, \ \  given by  
\[
    \Ini_{z;0,1}(u) = (u_0, u_1), \qquad u \in \GEVze{z},
\]
is a linear isomorphism. The analogous result is true for $\MGEVze{z}$ and $\left( \Mdc \right)^2$. In particular \ 
$\dim \GEVze{z} = 2d$ \ and \ $\dim \MGEVze{z} = 2d^2$.
\end{fact}

However, it is often more convenient to use another kind of ``initial conditions'', namely ``at $-1$ and $0$'' instead of $0$ and $1$. To formulate this properly,  we shall define first the appropriate extension of each solution (in both, vector and matrix cases), which is tightly related to the  ``a priori choice'' of $A_{-1}$:
\beql{A-1}
    A_{-1} := -\Id.
\eeq
The informal idea of the extension  is simply to    ``extend to $n=0$''  the system \eqref{eq:1} (and  \eqref{eq:2'} 
analogously) and to  ``compute the ``value at $-1$'', using our choice made in \eqref{A-1}, 
i.e., we get \  ``$U_{-1} := (B_0 - z\Id) U_0 + A_0 U_1$''
for the matrix case. Unfortunately, as one can see, such a definition seems to depend explicitly on the parameter $z$, and not only on $U$. So,  at the first sight it seems that  the notation for the extension ``to $-1$'' of a solution $U$ has to contain always this parameter, which would be  not very convenient. And  it is true, that we  really have this problem,  extending in such a way {\bf any} sequence $U \in \lnz{\Mdc}$ (similarly for the vector version). So we define first the following family
$\{\extlz{z}{} \}_{z \in \CC}$
of  ``extending transformations'' $\extlz{z}{}: \lnk{0}{\Mdc} \strz \lnk{-1}{\Mdc}$ given for $U \in \lnz{\Mdc}$   and $z \in \CC$ simply by
\beql{exttra-z}
    (\extlz{z}{U})_n := \opcje{(B_0 - z\Id) U_0 + A_0 U_1}{n=-1}{U_n}{n \in \Nz.}
\eeq
We shall use here the same notation for the vector sequences without any risk of confusion,  i.e., we shall also write $\extlz{z}{u}$ for $u \in \lnz{\CC^d}$   and $z \in \CC$ with the analogous meaning
\beql{exttra-z'}
    (\extlz{z}{u})_n := \opcje{(B_0 - z\Id) u_0 + A_0 u_1}{n=-1}{u_n}{n \in \Nz.}
\eeq
Therefore, for any $z\in\CC$ both kinds of transformations are  linear, and  moreover:
\beql{RHSconserv}
    \extlz{z}{(UC)} = (\extlz{z}{U}) C, \quad\mbox{for any $U \in \lnz{\Mdc}$, \ $C \in \Mdc$.}
\eeq

Fortunately, the situation is much simpler if  restricted 
to the subspace of all the (M)GEV-s. 
Namely, we have:
\begin{fact}\label{jed-ext}
If $z, w \in \CC$, $z \neq w$, then\footnote{Below $0$ denotes the zero sequence both for the $\CC^d$-vector and for the  matrix case.}
\[
    \GEVze{z} \cap \GEVze{w} = \{ 0 \}, \quad
    \MGEVze{z} \cap \MGEVze{w} = \{ 0 \}.
\]
\end{fact}
\begin{proof}
For the matrix case, consider $U$ satisfying  both \eqref{eq:1}, and its analog for $w$. Then, subtracting, we get 
\[
    0 = (z - w) U_n, \quad n \geq 1,
\]
hence $U_n = 0$ for any $n \in \N$, but now again  by \eqref{eq:1}  used only for $n=1$ we get also $U_0 = 0$, i.e, $U = 0$.
\end{proof}
This means, that for any non-zero vector or matrix solution of our equations, the parameter '$z$' is in fact ``coded in the solution''.  
On the other hand, independently of $z$, the value of the extension \ $\extlz{z}{}$ \ for  the zero sequence is obviously  also the zero sequence (on $\NN_{-1}$ already) by the linearity.  Hence denote
\[
    \GGEVze := \bigcup_{z\in\CC}\GEVze{z}, \quad 
    \MMGEVze := \bigcup_{z\in\CC}\MGEVze{z},  
\]
and, thanks to Fact~\ref{jed-ext}, 
for any $U \in \MMGEVze \setminus \{0\}$ \  ($u \in \GGEVze \setminus \{ 0 \}$) we can define $\Para(U)$ \  ($\Para(u)$) as the unique number $z \in \CC$ satisfying
\[
    U \in \MGEVze{z} \ \  (u \in \GEVze{z}).
\]
Finally, we can simplify our notation and omit the parameter '$z$', defining 
\[
    \extl{\hspace{ 0 ex}} : \MMGEVze \strz \lnk{-1}{\Mdc}
\]
given for $U \in \MMGEVze$   simply by  the formula
\beql{kropa}
    \extl{U} := \opcje{\extlz{\Para(U)}U}{U \neq 0}{0}{U = 0}
\eeq
and analogically for $u \in \GGEVze$.

Now, taking \eqref{A-1} into account, let us consider ``extensions''  of the systems  \eqref{eq:2'} and \eqref{eq:1}:
\beql{eq:2'-1}
    A_{n-1}^* u_{n-1} + B_n u_n + A_n u_{n+1} = z u_n, \qquad n \geq {\pmb 0}
\eeq
for   sequences $u=(u_n)_{n \in \Nk{-1}} \in \lnk{-1}{\CC^d}$ and 
\beql{eq:1-1}
    A_{n-1}^* U_{n-1} + B_n U_n + A_n U_{n+1} = z U_n, \qquad n \geq {\pmb 0}
\eeq
for   sequences $U=(U_n)_{n \in \Nk{-1}} \in \lnk{-1}{\Mdc}$. 
Their solutions will be called {\em extended
 generalized eigenvectors} and  
{\em extended 
 matrix generalized eigenvectors}, respectively, {\em (for $J$ and $z$)} 
  ---  {\em ``egev''} and  {\em ``emgev''} for short.

We denote also 
\[
    \GEV{z} := \big\{ u \in \lnk{-1}{\CC^d}:\, \mbox{$u$ is  an  egev for $J$ and $z$} \big\}
\]
and 
\[
    \MGEV{z} := \big\{ U \in \lnk{-1}{\Mdc}:\, \mbox{$U$ is  an emgev for $J$ and $z$} \big\}.
\]
Let us formulate explicitly some simple relations between all the above ``extended''  and  ``non-extended'' notions and the $\extl$ transformation.
\begin{fact} \label{fact:ext}
For any $z\in\CC$ the following assertions hold:
\begin{enumerate}[label=(\roman*)]
    \item \label{fact:ext:1}
    \  $\restr{\extl{}}{\GEVze{z}}\ = \  \restr{\extlz{z}{}}{\GEVze{z}}$,
    
    \item \label{fact:ext:2}
    \  $\restr{\extl{}}{\GEVze{z}} : \GEVze{z}\strz \GEV{z}$  \   is a linear isomorphism between $\GEVze{z}$ \  and $\GEV{z}$,
    
    \item \label{fact:ext:3}
    $\left( \restr{\extl{}}{\GEVze{z}} \right)^{-1} u =  \restr{u}{\N_0}$ \ \ for any $u \in \GEV{z}$, 

    \item \label{fact:ext:4}
    for any  $( c_{-1}, c_0 )\in(\CC^d)^2$ there is a unique $u \in \GEV{z}$ with $u_{-1} = c_{-1}$, $u_0 = c_0$, 
\end{enumerate}
and their obvious reformulations for the matrix sequences variants are also true. 
\end{fact}

\begin{proof}
Let's check, e.g., the $\CC^d$ vector version. Observe that \ref{fact:ext:1} is in fact the definition of $\extl{}$. Using this and taking $u \in \GEVze{z}$,  $v \in \GEV{z}$,   we get obviously \ 
$
    \restr{(\extl{u})}{\N_0} = u
$ 
by  \eqref{exttra-z'}, but we get also 
$
    \extl{(\restr{v}{\N_0})} = v
$,
because $v$ satisfies \eqref{eq:2'-1} --- in particular for $n = 0$. Linearity is clear by the definition, so \ref{fact:ext:2} and \ref{fact:ext:3} hold. Part~\ref{fact:ext:4} follows directly from \eqref{eq:2'-1} and \eqref{zal-odw+sym}. 
\end{proof}

It can be easily  checked that (e)mgev-s and (e)gev-s are mutually related  in the following simple ways.

\begin{fact}\label{factsolve-ma-ve}
If $U \in \lnmj{\Mdc}$, $z \in \CC$,  then 
\begin{enumerate}[label=(\roman*)]
    \item $U$ is an emgev for $J$ and $z$  iff \ for any $j=1,\ldots, d$ the vector sequence $U^{\{j\}} := \big( U_n^{\{j\}} \big)_{n \in \Nmj} \in \lnmj{\CC^d}$ \ is an egev for $J$ and $z$; 
    
    \item If $U$ is an emgev for $J$ and $z$ then  for any $v \in \CC^d$  the vector sequence  $U v:=(U_n v)_{n \in \Nmj} \in \lnmj{\CC^d}$ \ is an egev for $J$ and $z$. 
\end{enumerate}
  The analog result holds for  sequences $U \in \lnz{\Mdc}$ and mgev-s and gev-s.
\end{fact}

\subsection{Matrix orthogonal polynomials and Jitomirskaya--Last type semi-norms
}
\label{subs-orth-pol}

According to  Fact~\ref{fact:ext}\  \ref{fact:ext:4} for the matrix case, for any $z \in \CC$ choose
$Q(z), P(z) \in \MGEV{z}$  corresponding to
\begin{equation}
    \label{eq:22}
    \begin{cases}
        Q_{-1}(z) = \Id \\
        Q_{0}(z) = 0,
    \end{cases} \qquad
    \begin{cases}
        P_{-1}(z) = 0 \\
        P_{0}(z) = \Id,
    \end{cases}
\end{equation}
with the following general  notation:
for any  sequence $U(p) = ((U(p))_n)_{n \in \Nk{k}}$ depending on an extra ``function variable type parameter''  $p$:
\beql{nota-z-n}
    U_n(p) := (U(p))_n, \quad n \in \Nk{k}
\eeq
for any  $p$. The two  sequences of functions 
$Q, P$   are the so-called {\em  second} and {\em  first kind matrix orthogonal polynomials}\footnote{More precisely, this name and the orthogonality property belong to    the appropriate two sequences  $(Q_n)_{n \in \N}$, $(P_n)_{n \in \Nz}$ of matrix valued polynomial  functions $Q_n, P_n$ on $\RR$ or on  $\CC$  with the   values at each $z$ given by above defined $Q_n(z), P_n(z)$.}.

We can also use ``the conditions in $0,1$'' instead of those ``in $-1, 0$'':
\begin{equation}
    \label{eq:22'}
    \begin{cases}
        Q_{0}(z) = 0 \\
        Q_{1}(z) = A_0^{-1},
    \end{cases} \qquad
    \begin{cases}
        P_{0}(z) = \Id \\
        P_{1}(z) = A_0^{-1} (z \Id -B_0).
    \end{cases}
\end{equation}

These two  special solutions  $Q(z)$ and $P(z)$ play an  important algebraic role in the linear space of all matrix solutions $\MGEV{z}$.

\begin{fact}\label{factsolve-ma}
Let $z\in\CC$. 
\begin{enumerate}[label=(\roman*)]
    \item \label{factsolve-ma:1}
    If $U \in \MGEV{z}$, then  for any $V \in \Mdc$ also   $UV = (U_nV)_{n \in \Nz} \in \MGEV{z}$. 
 
    \item \label{factsolve-ma:2}
    Each  $U \in \MGEV{z}$ has the form 
    \begin{equation} \label{eq:4}
	    U = P(z) S + Q(z) T,
    \end{equation}
    for a  unique pair $(S,T)$ of matrices from $\Mdc$. This pair is given by
    \begin{equation}
        \label{eq:22a}
        \begin{cases}
            S := U_0 \\
            T := U_{-1} = (B_0 - z\Id) U_0 + A_0 U_1.
        \end{cases}
    \end{equation}
    
    \item \label{factsolve-ma:3}
    For any $S\in\Mdc$   the matrix sequence $H := \restr{(P(z)S)}{\Nz}$ satisfies ``the formal matrix  eigenequation for $\calJ$ and  $z$'', namely: $H \in \MGEVze{z}$ \   and 
    \begin{equation} \label{eq:1dla0}
        B_0 H_0 + A_0 H_{1} = z H_0.
    \end{equation}
    So, for any $v \in \CC^d \setminus \{ 0 \}$ the $\CC^d$-sequence $h := \restr{(P(z)v)}{\Nz}$ is an eigenvector of the formal operator $\calJ$ for $z$:
    \beql{foreign}
        \calJ h = z h, \quad h \neq 0.    
    \eeq
\end{enumerate}
\end{fact}
\begin{proof}
Part~\ref{factsolve-ma:1} is obvious. Hence, using it, by linearity, by  \eqref{eq:22} and by the unicity from  Fact~\ref{fact:ext} \ref{fact:ext:4}, we get  \eqref{eq:4}  with $S$ and $T$  given by \eqref{eq:22a}. --- I.e., we  can simply assume that some $S$ and $T$ are  given by \eqref{eq:22a} and then we see that the  initial conditions of the solution on the  RHS of \eqref{eq:4} are just the pair $(U_{-1},U_0)$, which proves  \eqref{eq:4} by the unicity. So, to finish~\ref{factsolve-ma:2} we should check that the choice of the   pair $(S,T)$ is also  unique.  By the linearity, it suffices to check only that if $P(z) S + Q(z) T$ is the zero solution, then $S = T = 0$. Indeed, in this case we have \  $0 = P_0(z) S + Q_0(z) T = S$ and $0 = P_{-1}(z) S + Q_{-1}(z) T = T$. Now, to get \ref{factsolve-ma:3}, we can first use \ref{factsolve-ma:1} with \ref{factsolve-ma:2} for  $U$ of the form  \eqref{eq:4} with $T = 0$, so, using also Fact~\ref{fact:ext} (the matrix version), we get 
$H \in \MGEVze{z}$ with 
\eqref{eq:1dla0} obtained  by \eqref{eq:22a} for $T = 0$. Now we obtain $\left( \calJ h \right)_n = z h_n$
by \eqref{eq:1} for $n \geq 1$ and separately for $n = 0$ from \eqref{eq:1dla0} with $S = \Id$. Finally $h \neq 0$, because by \eqref{eq:22} \  $h_0 = (P(z) v)_0 = P_0 v = v \neq 0$.
\end{proof}

Assume now that $J$ is s.a. 
We show here  a result being a block case analog of the appropriate   scalar case one result from \cite{Jitomirskaya1999}. 
It makes  use of the special choice of Jitomirskaya--Last type semi-norms (see subsection \ref{J-Lsemi}) and, similarly as in the case $d=1$,  it can be useful   
to ``control  the boundary limits of the matrix Weyl function'' 
(see Definition~\ref{maWe}).

\begin{proposition} \label{defJL}
Suppose that $J$ is self-adjoint. For any  $\la \in \RR$  there exists a unique function \ \ $\JLf{\la} : \RR_+ \to \RR_+$ \  satisfying 
\begin{equation} \label{eq:13}
	\norm{P(\la)}_{[0, \JLf{\la}(\epsilon)]} 
	\norm{Q(\la)}_{[0, \JLf{\la}(\epsilon)]} = \frac{1}{2 \epsilon},\quad \eps>0.
\end{equation}
 Moreover, $\JLf{\la}$ is a strictly decreasing continuous function and satisfies
\begin{equation} \label{eq:106a}
    \lim_{\epsilon \to 0^+} \JLf{\la}(\epsilon) = +\infty, \quad
    \lim_{\epsilon \to +\infty} \JLf{\la}(\epsilon) = 0.
\end{equation}
Consequently, its inverse $\JLf{\la}^{-1}$ is also strictly decreasing continuous function and satisfies
\begin{equation} \label{eq:106b}
    \lim_{t \to 0^+} \JLf{\la}^{-1}(t) = +\infty, \quad
    \lim_{t \to +\infty} \JLf{\la}^{-1}(t) = 0.
\end{equation}
\end{proposition}

\begin{proof}
Define a function $f : \RR_+ \to \RR_+$ by the formula
\[
    f(t) = 
    \norm{P(\la)}_{[0,t]} 
    \norm{Q(\la)}_{[0,t]}.
\]
By \eqref{eq:105a} this function is continuous, non-negative and weakly increasing. Moreover, by \eqref{eq:22} it is positive  in fact. Let us observe that it is strictly increasing. Indeed, because if not, then there would exist $n \in \Nz$ such that both $\norm{P(\la)}_{[0,t]}$ and $\norm{Q(\la)}_{[0,t]}$ were constant for $t \in (n, n+1)$. It would mean that $\norm{P_{n+1}(\la)} = \norm{Q_{n+1}(\la)} = 0$. Which by \eqref{eq:84} would imply that $R_n(\la)$ was singular. This contradicts \eqref{eq:36}. Next, observe that 
\[
    \lim_{t \to 0^+} f(t) = 0, \quad
    \lim_{t \to +\infty} f(t) = +\infty.
\]
Indeed, the first limit follows from $\norm{Q_0(\la)} = 0$. The second follows from the fact that if $\norm{P(\la)}_{[0,+\infty]} < +\infty$ and $\norm{Q(\la)}_{[0,+\infty]} < +\infty$, then the operator $J$ is not self-adjoint, see \cite[Theorem 1.3]{Dyukarev2020}. Thus, we have shown that $f$ is continuous, strictly increasing and surjective. So its inverse $f^{-1} : \RR_+ \to \RR_+$ has the same properties. Consider the function $g : \RR_+ \to \RR_+$ defined by $g(\epsilon) = 1/(2\epsilon)$. This function is strictly decreasing and surjective. Thus, if $\JLf{\la}$ function exists it is a solution of the equation
\[
    f \big( \JLf{\la}(\epsilon) \big) = g(\epsilon).
\]
This equation has a unique solution given by
\[
    \JLf{\la}(\epsilon) = f^{-1} \big( g(\epsilon) \big).
\]
It is immediate that this defines a function satisfying \eqref{eq:13}. From this representation it is immediate that $\JLf{\la} : \RR_+ \to \RR_+$ is a continuous strictly decreasing surjective function. It implies \eqref{eq:106a}. Since again $\JLf{\la}^{-1}$ has analogous properties, we also obtain \eqref{eq:106b}.
\end{proof}

The unique function $\JLf{\la}$ described above  will be called \emph{J-L function (for $J$ and $\la$)}. Note that $\JLf{\la}$ and Proposition \ref{defJL} were essentially used to prove the main result of our paper \cite{BarrierNonsubordinacy}.

\subsection{Transfer matrices and the  Liouville--Ostrogradsky formulae}
\label{subs-trans}
In the scalar case $d=1$ 
the transfer matrix sequences turned out to be a very useful tool for describing  spectral properties of the operator $J$. As we shall see, this is  the case of general dimension $d$.

Let us fix here $z \in \CC$.  Our  basic difference equations:  the  generalized eigenequation \eqref{eq:2},  its matrix analog \eqref{eq:1},  as well as their extended variants, can be  written in  equivalent forms with the use of the so-called (one step) \emph{transfer matrices (for $J$ and $z$)}. The $n$th transfer matrix  $T_n(z) \in \Mddc$ has  the block form, with blocks in $\Mdc$: 
\begin{equation}
    \label{eq:21}
    T_n(z):=
    \begin{pmatrix}
        0 & \Id \\
        -A_{n}^{-1} A_{n-1}^* & A_n^{-1} (z \Id - B_n)
    \end{pmatrix},  \qquad n \geq 0 
\end{equation}
(for $n = 0$ recall that  $A_{-1} = -\Id$ by \eqref{A-1}).  Hence, obviously, \eqref{eq:2'} \  (\eqref{eq:2'-1}) is equivalent to 
\begin{equation}
    \label{eq:20'}
    \begin{pmatrix}
        u_n \\
        u_{n+1}
    \end{pmatrix} 
    =
    T_n(z) 
    \begin{pmatrix}
        u_{n-1} \\
        u_n
    \end{pmatrix}, \qquad n \geq 1  \ (\geq 0).
\end{equation} 
Similarly,  \eqref{eq:1} (\eqref{eq:1-1}) is equivalent to
\begin{equation}
    \label{eq:20}
    \begin{pmatrix}
        U_n \\
        U_{n+1}
    \end{pmatrix} 
    =
    T_n(z) 
    \begin{pmatrix}
        U_{n-1} \\
        U_n
    \end{pmatrix}, \qquad n \geq 1 \ (\geq 0).
\end{equation}
Let us observe that  $T_n(z)$ is invertible and 
\begin{equation} \label{eq:35}
    \big( T_n(z) \big)^{-1}=
    \begin{pmatrix}
        \big( A_{n-1}^* \big)^{-1} (z \Id - B_n) & -\big( A_{n-1}^* \big)^{-1} A_n \\
        \Id & 0
    \end{pmatrix}, \qquad n \geq 0,
\end{equation}
which is clear by  direct multiplying (or by expressing $u_{n-1}$ by $u_n$ and  $u_{n-1}$ from  \eqref{eq:2'-1}). 

Moreover we define {\em the $n$-step transfer matrix} by 
\begin{equation} \label{eq:55}
    R_n(z) = T_{n-1}(z)\ldots T_0(z),  \qquad n\geq 1.
\end{equation}
This name is justified, e.g., by the property
\begin{equation} \label{eq:56}
    \begin{pmatrix}
        U_{n-1} \\
        U_{n}
    \end{pmatrix}
    =
    R_n(z) 
    \begin{pmatrix}
        U_{-1} \\
        U_0
    \end{pmatrix},
    \quad n \geq 1, 
\end{equation}
which we obtain  from \eqref{eq:20}. Hence, by   \eqref{eq:56} and \eqref{eq:22} we get
\begin{equation} \label{eq:84}
    R_n(z) =
    \begin{pmatrix}
        Q_{n-1}(z) & P_{n-1}(z) \\
        Q_n(z) & P_n(z)
    \end{pmatrix}, \qquad n \geq 1,
\end{equation}
being simply a direct consequence of  $R_n(z)=R_n(z)\begin{pmatrix}
        \Id & 0 \\
        0 & \Id
    \end{pmatrix}$.

Presently we shall derive a formula for the inverse of $R_n(z)$ expressing it explicitly in terms of $R_n(\overline{z})$.  
Let us set
\begin{equation} \label{eq:33}
    K_n :=
    \begin{pmatrix}
        A_n^* & 0 \\
        0 & \Id
    \end{pmatrix},  \qquad n \geq -1
\end{equation}
and
\begin{equation} \label{eq:32}
    \tilde{T}_n(z) := K_n T_n(z) K_{n-1}^{-1}, \qquad n \geq 0.
\end{equation}
So by \eqref{eq:21}
\[
    \tilde{T}_n(z) =
    \begin{pmatrix}
        0 & A_n^* \\
        -A_{n}^{-1} & A_n^{-1} (z \Id - B_n)
    \end{pmatrix}, \qquad n \geq 0 
\]
Therefore, defining 
\[
    \Omega :=
    \begin{pmatrix}
        0 & \Id \\
        -\Id & 0
    \end{pmatrix},
\]
we verify by direct computations that 
\begin{equation} \label{eq:31} 
    \Omega = \big( \tilde{T}_n(\overline{z}) \big)^* \Omega \tilde{T}_n(z), \quad n \geq 0.
\end{equation}
By using \eqref{eq:32} we get
\begin{equation} \label{eq:34}
    R_n(z) = 
    (K_{n-1}^{-1} \tilde{T}_{n-1}(z) K_{n-2}) (K_{n-2}^{-1} \tilde{T}_{n-2}(z) K_{n-3}) \ldots (K_0^{-1} \tilde{T}_0(z) K_{-1})
    =
    K_{n-1}^{-1} \tilde{R}_n(z) K_{-1},
\end{equation}
where 
\begin{equation} \label{eq:31a}
    \tilde{R}_n(z) := \tilde{T}_{n-1}(z) \tilde{T}_{n-2}(z) \ldots \tilde{T}_0(z), \quad n \geq 1.
\end{equation}
We claim that
\begin{equation} \label{eq:31b}
    \Omega = \big( \tilde{R}_n(\overline{z}) \big)^* \Omega \tilde{R}_n(z), \quad n \geq 1.
\end{equation}
We shall prove it inductively. By \eqref{eq:31a} we have $\tilde{R}_1(z) = \tilde{T}_0(z)$ for any $z \in \CC$. Thus, in view of \eqref{eq:31} the formula \eqref{eq:31b} holds true for $n=1$. Next, if \eqref{eq:31b} holds for some $n \geq 1$, then by \eqref{eq:31} we have
\begin{align*}
    \Omega 
    &= 
    \big( \tilde{R}_n(\overline{z}) \big)^* \Omega \tilde{R}_n(z) \\
    &=
    \big( \tilde{R}_n(\overline{z}) \big)^* 
    \Big( \big( \tilde{T}_n(\overline{z}) \big)^* \Omega \tilde{T}_n(z) \Big)
    \tilde{R}_n(z) \\
    &=
    \big( \tilde{R}_{n+1}(\overline{z}) \big)^* \Omega \tilde{R}_{n+1}(z),
\end{align*}
where in the last equality we have used \eqref{eq:31a}. It ends the inductive step in the proof of \eqref{eq:31b}. Thus, by multiplying both sides of \eqref{eq:31b} by $\Omega^{-1}$ on the left and then by $\big(\tilde{R}_n(z) \big)^{-1}$ on the right we arrive at
\[
    \big( \tilde{R}_n(z) \big)^{-1} = \Omega^{-1} \big( \tilde{R}_n(\overline{z}) \big)^* \Omega.
\]

Consequently, using \eqref{eq:34} twice   we can derive
\[
    \big( R_n(z) \big)^{-1} = 
    K_{-1}^{-1}\Omega^{-1} \big( K_{-1}^{-1} \big)^*
    \big( R_n(\overline{z}) \big)^* K_{n-1}^* \Omega K_{n-1},
\]
so, by \eqref{eq:33}, finally
\begin{equation} \label{eq:36}
    \big( R_n(z) \big)^{-1} =
    \begin{pmatrix}
        0 & \Id \\
        -\Id & 0
    \end{pmatrix}
    \big( R_n(\overline{z}) \big)^*
    \begin{pmatrix}
        0 & A_{n-1} \\
        -A_{n-1}^* & 0
    \end{pmatrix}, 
    \quad z \in \CC.
\end{equation}
Note that \eqref{eq:36} has been obtained by using the same argument  in \cite[formula (5)]{Janas2014} for real  $z$ and with real  $A_n$-s and $B_n$-s.

The following result is well-known, see e.g. \cite[Theorem 5.2]{Berg2008} and \cite[Lemma 2.4]{Oliveira2021a}, however our proof seems to be new and, even taking the above preparatory calculations into account, it is relatively short.

\begin{theorem}[Liouville--Ostrogradsky]\label{LO}
For any $w \in \CC$ one has
\begin{align}
    \label{eq:LO1}
    &Q_k(w) \big( P_k(\overline{w}) \big)^* = P_k(w) \big( Q_k(\overline{w}) \big)^*, \quad k \geq 0 \\
    \label{eq:LO2}
    &Q_{k}(w) \big( P_{k-1}(\overline{w}) \big)^* - P_{k}(w) \big( Q_{k-1}(\overline{w}) \big)^* = A_{k-1}^{-1}, \quad k \geq 1.
\end{align}
\end{theorem}
\begin{proof}
By \eqref{eq:36} and \eqref{eq:84} we have
\[
    \begin{pmatrix}
        \Id & 0 \\
        0 & \Id
    \end{pmatrix}
    = 
    R_k(w) R_k^{-1}(w)
    =
    \begin{pmatrix}
        -P_{k-1}(w) & Q_{k-1}(w) \\
        -P_k(w) & Q_k(w)
    \end{pmatrix}
    \begin{pmatrix}
        -\big( Q_{k}(\overline{w}) \big)^* A_{k-1}^* & \big( Q_{k-1}(\overline{w}) \big)^* A_{k-1} \\
        -\big( P_{k}(\overline{w}) \big)^* A_{k-1}^* & \big( P_{k-1}(\overline{w}) \big)^* A_{k-1}
    \end{pmatrix}.
\]
Thus, the formulas \eqref{eq:LO1} and \eqref{eq:LO2} follows from computing the last row.
\end{proof}

Now we are ready  to solve the non-homogeneous version of the matrix recurrence relation \eqref{eq:1-1}. The fact that despite of the general  non-commutativity  of matrices of size $d>1$ we obtain the formula for the solution,  which is similar to the one  from the scalar ($d=1$) case, is a kind of a miracle.
And the reason for this miracle lies  precisely in the  Liouville--Ostrogradsky formulae. This result is  also  not new, see e.g. \cite[Section 2]{Kostyuchenko1998},  but 
we present our proof  simpler and more detailed than other proofs we found in the literature.

\begin{proposition}
Let $F \in \lnz{\Mdc}$. The unique solution  $S=(S_n)_{n \in \Nk{-1}} \in \lnk{-1}{\Mdc}$  of the recurrence relation
\begin{equation} \label{eq:59}
    A_n S_{n+1} + B_n S_n + A_{n-1}^* S_{n-1} = z S_n + F_n, \quad n \geq 0.
\end{equation}
with the initial conditions $S_{-1} = S_0 = 0$ is equal to
\begin{equation} \label{eq:58}
    S_n := 
    \sum_{k=0}^{n-1} 
    \Big( Q_n(z) \big( P_k(\overline{z}) \big)^* - P_n(z) \big( Q_k(\overline{z}) \big)^* \Big) F_k, \quad n \geq -1.
\end{equation}
\end{proposition}
\begin{proof}
Since all $A_k$ are invertible it is clear that the solution of \eqref{eq:59} with given initial conditions $S_{-1}$ and $S_0$ is unique. It remains to prove that it is  satisfied by the sequence defined by   \eqref{eq:58}.

It is immediate from \eqref{eq:58} that $S_{-1} = S_0 = 0$. Moreover, we get
\[
    S_1 = \Big( Q_1(z) \big( P_0(\overline{z}) \big)^* - P_1(z) \big( Q_0(\overline{z}) \big)^* \Big) F_0 = A_0^{-1} F_0
\]
which is in agreement with \eqref{eq:59} for $n=0$.
So let us assume that $n \geq 1$. Since both $P(z)$ and $Q(z)$ satisfy \eqref{eq:1} we get
\begin{align*}
    A_n S_{n+1} 
    &=
    A_n 
    \sum_{k=0}^{n} 
    \Big( Q_{n+1}(z) \big( P_k(\overline{z}) \big)^* - P_{n+1}(z) \big( Q_k(\overline{z}) \big)^* \Big) F_k \\
    &=
    (z \Id - B_n) 
    \sum_{k=0}^{n} \Big( Q_{n}(z) \big( P_k(\overline{z}) \big)^* - P_{n}(z) \big( Q_k(\overline{z}) \big)^* \Big) F_k \\
    &- A_{n-1}^*
    \sum_{k=0}^{n} \Big( Q_{n-1}(z) \big( P_k(\overline{z}) \big)^* - P_{n-1}(z) \big( Q_k(\overline{z}) \big)^* \Big) F_k.
\end{align*}
Thus,
\[
    A_n S_{n+1} = (z \Id - B_n) S_n - A_{n-1}^* S_{n-1} + \tilde{F}_n,
\]
where
\begin{align}
    \label{eq:60}
    \tilde{F}_n &=
    (z \Id - B_n) \Big( Q_{n}(z) \big( P_n(\overline{z}) \big)^* - P_{n}(z) \big( Q_n(\overline{z}) \big)^* \Big) F_n \\
    \nonumber
    &-
    A_{n-1}^* \Big( Q_{n-1}(z) \big( P_n(\overline{z}) \big)^* - P_{n-1}(z) \big( Q_n(\overline{z}) \big)^* \Big) F_n \\
    \nonumber
    &-
    A_{n-1}^* \Big( Q_{n-1}(z) \big( P_{n-1}(\overline{z}) \big)^* - P_{n-1}(z) \big( Q_{n-1}(\overline{z}) \big)^* \Big) F_{n-1}.
\end{align}
It remains to prove that $\tilde{F}_n = F_n$. To do so, let us 
apply \eqref{eq:LO1} for $w=z$ with $k=n$ and $k=n-1$. Then we get that on the right-hand side of \eqref{eq:60} the first and the third lines are equal to $0$. By considering \eqref{eq:LO2} for $w = \overline{z}$ with $k=n$ and taking the adjoint of both sides we get
\begin{align*}
    \big( A_{n-1}^{-1} \big)^*
    &= 
    \Big( Q_{n}(\overline{z}) \big( P_{n-1}(z) \big)^* - 
    P_{n}(\overline{z}) \big( Q_{n-1}(z) \big)^* \Big)^* \\
    &= 
    P_{n-1}(z) \big( Q_n(\overline{z}) \big)^* - Q_{n-1}(z) \big( P_n(\overline{z}) \big)^*,
\end{align*}
which results in  the second line on the right-hand side of \eqref{eq:60} being  equal to $F_n$, and consequently, $\tilde{F}_n = F_n$. It ends the proof.
\end{proof}

\section{The Weyl function} \label{sec:4} 

Similarly to the scalar Jacobi case,  Weyl coefficient (being  a matrix for  the block case), is the main object in the method of  subordinacy, which gives  the link between generalized eigenvectors and the absolute continuous and the singular part of the spectral measure. And consequently -- the absolutely continuous  and the singular  spectrum of $J$.

\subsection{$\ell^2$ matrix solutions and the matrix Weyl function  $W$}\label{l2mW}

In the scalar Jacobi case  ``the scalar orthogonal polynomials'' are used, with the common notation $p(z) := P(z), \ q(z) := Q(z)$ for  $z \in \CC$. That is, since $d = 1$, we  treat complex numbers as  elements of $\Mdc$ and also as  $\CC^d$-vectors, and    $p(z), q(z)$ are solutions of   both  \eqref{eq:2'-1} and  \eqref{eq:1-1}, being now  just the same equation. 

It is also well-known for this case that if $J$ is s.a. and  $z \in \CC \setminus \RR$, then neither $\restr{p(z)}{\N_0}$, nor $\restr{q(z)}{\N_0}$  belong to the Hilbert space $\ldnz{\CC}$ in which Jacobi operator $J$ acts, and   there exists exactly one $w(z) \in \CC$ such that \ 
$\restr{\big( w(z) p(z) + q(z) \big)}{\N_0} \in \ldnz{\CC}$. Surely, instead of making the restriction to $\N_0$, we can equivalently just claim here that 
$p(z), q(z) \notin \ldnmj{\CC}$, and
$w(z) p(z) + q(z) \in \ldnmj{\CC}$, respectively.
The above unique $w(z)$ is called {\em the Weyl coefficient (for $J$ and $z$)}, and the appropriate function 
$w: \CC \setminus \RR \strz \CC$ is called {\em the Weyl function for $J$}.

Let us recall here some less known generalisations of the above  results and definitions for block Jacobi case. 
Hence, assume temporarily that $J$ is s.a. and $z \in \CC \setminus \RR$, 
until we give up these assumptions in the last part of this subsection.
Thus $z \notin \si(J)$, and we can define:
\beql{ujz}
    u^{(j)}(z) := (J - z\Id)^{-1} \de_0(e_j), \ \ j=1,\ldots, d.
\eeq
In particular, for any $j$ we have $u^{(j)}(z) \in \Dom(J) \subset \ldnz{\CC^d}$ and 
\[
    \calJ u^{(j)}(z) = z u^{(j)}(z) + \de_0(e_j), \ \ j=1,\ldots, d.
\]
Considering the  terms $n \geq 1$ of the above equality of sequences we see that  each $u^{(j)}(z)$ is a gev for $J$ and $z$. Moreover,  when $n=0$ we get 
\beql{eq-for0}
    e_j = \left( (\calJ - z\Id) u^{(j)}(z) \right)_0 = 
    (B_0 - z \Id) u^{(j)}_0(z)+ A_0 u^{(j)}_1(z).
\eeq
So,  defining  matrix sequences
\beql{def-U(z)}
    \tilde{U}(z) := 
    [u^{(1)}(z), \ldots, u^{(d)}(z)] \qquad  \mbox{and} \qquad  \quad 
    U(z) :=\ \extl{(\tilde{U}(z))},
\eeq
we have 
\beql{obaUinltwo}
\tilde{U}(z)\in\ldnz{\Mdc}, \qquad U(z) \in \ldnmj{\Mdc} 
\eeq 
and by  Fact~\ref{factsolve-ma-ve} (the version for gev-s and mgev-s) together with  Fact~\ref{fact:ext} we see that $\tilde{U}(z)$ is an mgev for $J$ and $z$ and  $U(z)$ is an emgev for $J$ and $z$.  
\begin{definition}
If   $J$ is s.a.  and $z \in \CC \setminus \RR$, then   $U(z)$ \ is  {\em the Weyl matrix  solution for $J$ and $z$}.  
\end{definition}

By \eqref{exttra-z} and \eqref{eq-for0} we obtain 
\beql{eq-matrfor0}
    \Id = 
    (B_0 - z\Id) U_0(z)+ A_0 U_1(z) = 
    U_{-1}(z).
 \eeq
For each $z\in\CC \setminus \RR$ and $\nu\in\CC^d$ denote
\beql{unuz}
u(\nu, z):=\tilde{U}(z)\nu.
\eeq
Soon it will be convenient for us to use the following simple results. 

\begin{lemma}\label{lemunuz}
If   $J$ is s.a., \ then for any $\nu\in\CC^d$ and  $z \in \CC \setminus \RR$
\beql{lemunuzi} u(\nu, z)\in D(J);
\eeq
\beql{lemunuzii} u(\nu, z)= (J-zI)^{-1}\de_{0}(\nu);
\eeq
\beql{lemunuziii} u(e_j, z)=u^{(j)}(z), \ \ j=1,\ldots, d;
\eeq
\beql{lemunuziv}
  \lesc \nu, u_0(\nu, z) \risc_{\CC^d}= 
    \lesc Ju(\nu, z) , u(\nu, z) \risc_{\el^2} - z\norm{u(\nu, z)}^2_{\el^2}.  
\eeq
\end{lemma}
\begin{proof}
By the linearity of $\tilde{U}(z)$ and by \eqref{def-U(z)} 
$$
u(\nu, z):=\tilde{U}(z)\nu=\sum_{j=1}^d\nu_j\tilde{U}(z)e_j=\sum_{j=1}^d\nu_ju^{(j)}(z)
$$
which gives \eqref{lemunuzi}, because 
each $u^{(j)}(z)$ is  in $\Dom(J)$. 
Now, by \eqref{ujz}  and by  the linearity of $(J-zI)^{-1}$
we have
$$
\sum_{j=1}^d\nu_ju^{(j)}(z)=\sum_{j=1}^d\nu_j(J-zI)^{-1}\de_{0}(e_j)= (J-zI)^{-1}\de_{0}\left(\sum_{j=1}^d\nu_je_j\right)=(J-zI)^{-1}\de_{0}(\nu),
$$
hence  we obtain \eqref{lemunuzii}, which gives also \eqref{lemunuziii}, by \eqref{ujz}.
Using \eqref{lemunuzii} we get 
$$
\de_{0}(\nu)=(J-zI)u(\nu, z)=Ju(\nu, z)-zu(\nu, z),
$$
and taking the scalar product of booth sides with $u(\nu, z)$ we obtain
$$
\lesc Ju(\nu, z) , u(\nu, z) \risc_{\el^2} - z\norm{u(\nu, z)}^2_{\el^2}= \lesc \de_{0}(\nu) , u(\nu, z) \risc_{\el^2}= \lesc \nu, u_0(\nu, z) \risc_{\CC^d},
$$
by \eqref{den}, so \eqref{lemunuziv} is proved.
\end{proof}

We can now formulate the expected result on ``matrix $\ell^2$  solutions''. For some special cases it  is not new --- for the case $A_n = A_n^*$  
see, e.g., \cite[Theorem VII.2.8]{Berezanskii1968} and for the  case $A_n \equiv \Id$ see \cite[Section 2]{Acharya2019}. We give here a simple proof for the general case, for  the sake of self-sufficiency.

\begin{proposition}
    \label{prop-ltwosol}
Suppose that $J$ is s.a. and $z \in \CC \setminus \RR$. Then
there exists exactly one $\Wej(z)\in\Mdc$ such that  
\beql{eq-ltwo}
    P(z) \Wej(z) + Q(z) \in \ldnmj{\Mdc}.
\eeq
Moreover, with  the above  unique $\Wej(z)$ 
\begin{enumerate}[label=(\roman*)]
    \item \label{eq:ltwosol:1}
    $P(z) \Wej(z) + Q(z)$ \  is the Weyl  matrix  solution for $J$ and $z$:
    \beql{eq-Weylmatr}
        P(z) \Wej(z) + Q(z) = U(z);
    \eeq
    
    \item \label{eq:ltwosol:2}
    $\Wej(z)=\tilde{U}_0(z) = U_0(z)$,
    
    \item \label{eq:ltwosol:3}
    $\det \Wej(z) \neq 0$.
\end{enumerate}
\end{proposition}
\begin{proof}
Fix $z \notin \RR$ and consider the  Weyl matrix  solution $U(z)$ for $J$ and $z$. By Fact~\ref{factsolve-ma} \ref{factsolve-ma:2} $U(z)$ has the form \eqref{eq:4}, where $S = U_0(z)=\tilde{U}_0(z)$ and $T = U_{-1}(z) = \Id$,  by \eqref{eq-matrfor0}.  So \eqref{obaUinltwo}   proves the ``exists''-- part of the assertion and, assuming that we already have the  uniqueness, we will also get  \ref{eq:ltwosol:1} and \ref{eq:ltwosol:2}. 
Hence, let us  now  prove  the uniqueness. 
Suppose that for some $z \notin \RR$ there exist two different matrices ``$\Wej(z)$'' satisfying \eqref{eq-ltwo}. Then, subtracting, we get a non-zero $C \in \Mdc$ such that \ 
$P(z) C \in \ldnmj{\Mdc}$. Now, choosing $w \in \CC^d$ such that $v := C w \neq 0$ we get 
$P(z) \nu \in \ldnmj{\CC^d}$. Thus, using Fact~\ref{factsolve-ma} \ref{factsolve-ma:3},  for $h := \restr{(P(z)\nu)}{\N_0}$ we get  $\calJ h = z h \in \ldnz{\CC^d}$, which means that $h \in \Dom(J)$. Moreover $h \neq 0$, because $h_0 = P_0(z) \nu = \Id \nu = \nu \neq 0$. Thus $h$ is an eigenvector of $J$ with the eigenvalue $z \notin \RR$ --- a contradiction with the assumption, that $J$ is s.a.

To prove  \ref{eq:ltwosol:3}, i.e.,  that $\Wej(z)$ is invertible,  consider $\nu \in \Ker W(z)$, and the vector $w := u(\nu,z)$.
By \eqref{unuz} and by \ref{eq:ltwosol:2}, just proved, we get
$$
w_0 = U_0(z)\nu = W(z)\nu= 0,
$$
which together with     \ref{lemunuziv} gives 
$$
   0= \lesc \nu, w_0 \risc_{\CC^d}  = 
    \lesc Jw , w \risc_{\el^2} - z\norm{w}^2_{\el^2}.  
$$
Therefore, using the s.a. of $J$, we get \ 
$z \norm{w}^2_{\el^2} = \lesc Jw , w \risc_{\el^2} \in \RR$. 
But since $z \notin \RR$, $w$ has to be the $0$-vector.  Now, by

\eqref{lemunuzii}
$$
\de_{0}(\nu)=(J-zI)w=0,
$$
so $\nu=0$.
\end{proof}

 It is a good moment to note the following result related to \eqref{eq-ltwo}.
\begin{corollary}\label{cor-PQnie}
Suppose that $J$ is s.a. and $z\in\CC\setminus\RR$. Then
\[
    P(z), \ Q(z) \notin \ldnmj{\Mdc}.
\]
\end{corollary}

\begin{proof} We have $Q(z) = P(z) 0 + Q(z)$, so if  $ Q(z) \in \ldnmj{\Mdc}$ then $\Wej(z) = 0$ by the uniqueness from Proposition~\ref{prop-ltwosol}. But it contradicts the condition $\det \Wej(z) \neq 0$ from \ref{eq:ltwosol:3}.

By \eqref{eq-ltwo} and \ref{eq:ltwosol:3} of Proposition \ref{prop-ltwosol} we get 
$P(z)+ Q(z)(\Wej(z))^{-1} \in \ldnmj{\Mdc}$. Thus, if $P(z) \in \ldnmj{\Mdc}$ then also $Q(z)(\Wej(z))^{-1} \in \ldnmj{\Mdc}$.
But then also $Q(z) = \left( Q(z)(\Wej(z))^{-1} \right) \Wej(z) \in \ldnmj{\Mdc}$, which contradicts the part just proved.
\end{proof}

 Thanks to the results of Proposition \ref{prop-ltwosol},  the notion of Weyl coefficient can be  generalized from the ''scalar'' Jacobi operator case  to  all the block Jacobi operators with any (finite)  dimension $d$ of blocks.

\begin{definition}\label{maWe}
Let $J$ be s.a. For fixed $z \in \CC \setminus \RR$\ 
such   $\Wej(z)$,  that \eqref{eq-ltwo} holds is called {\em the matrix Weyl coefficient  (for $J$ and $z$)},
and the appropriate function 
$W : \CC \setminus \RR \strz \Mdc$ is called {\em the matrix Weyl function (for $J$)}.\footnote{Using the  argumentation from  the proof of   Proposition~\ref{prop-ltwosol} and from the beginning of this subsection one can easily  see that in fact it suffices here to assume that $z \in \CC \setminus \si(J)$ to  properly define the matrix Weyl coefficient  for $J$ and $z$. But note also  that  the invertibility of $\Wej(z)$ from  property~\ref{eq:ltwosol:3} is guaranteed only for  $z \in \CC \setminus \RR$.}
\end{definition}

We usually omit here the dependence on $J$ in the notation and we write simply $W$, assuming that we consider a fixed $J$.

\begin{theorem} \label{prop:9ogolniejszy}
Assume that $J$ is  s.a. and  $z \in \CC \setminus \RR$. For any $\nu \in \CC^d$ 
\begin{equation} \label{eq:114a}
    \frac{1}{\Im z}
    \big\langle \big( \Im \Wej(z) \big) \nu, \nu \big\rangle_{\CC^d}
    =\norm{\tilde{U}(z) \nu}^2_{\el^2}=
    \norm{ U(z) \nu}_{[0, +\infty]}^2\geq
    \norm{ \Wej(z) \nu}^2.
\end{equation}
Moreover, $\Im \Wej(z)$ is strictly positive for $\Im z>0$ and strictly negative for $\Im z<0$ and 
\begin{equation} \label{eq:114b}
\norm{\tilde{U}(z)}^2_{\ldnz{\Mdc}}=
    \norm{U(z)}_{[0, +\infty]}^2 \leq
	\frac{\tr \big( \Im \Wej(z) \big)}{\Im z}.
\end{equation}
\end{theorem}
\begin{proof}
The equality \ \ $\norm{\tilde{U}(z) \nu}^2_{\el^2}=
    \norm{ U(z) \nu}_{[0, +\infty]}^2$ 
is obvious by the definition of $U(z)$ (see \eqref{def-U(z)}) and by 
\eqref{eq:105b}. Taking the imaginary parts of both sides of \eqref{lemunuziv} and using Proposition \ref{prop-ltwosol} \ref{eq:ltwosol:2}, \eqref{unuz}  
  and the s.a. of $J$ we get
  \begin{align}
\label{imoba}
\Im\big\langle  \nu, \Wej(z) \nu \big\rangle _{\CC^d}=
\Im\big\langle  \nu, \tilde{U}_0(z) \nu \big\rangle _{\CC^d}&=\Im\big\langle  \nu, u_0(\nu, z) \big\rangle _{\CC^d}\\&= 
     - \Im z\norm{u(\nu, z)}^2_{\el^2}=- \Im z\norm{\tilde{U}(z) \nu}^2_{\el^2}.\nonumber
\end{align}
By Proposition \ref{prop:18} \ref{prop:18:reimsc} we have also 
$$
\Im\big\langle  \nu, \Wej(z) \nu \big\rangle _{\CC^d}=-\big\langle  \nu, (\Im\Wej(z)) \nu \big\rangle _{\CC^d}=-\big\langle  (\Im\Wej(z)) \nu,  \nu \big\rangle _{\CC^d},
$$
thus by \eqref{imoba} we obtain the equality from 
\eqref{eq:114a}, and the inequality follows from Proposition \ref{prop-ltwosol} \ref{eq:ltwosol:2}. 

To get the assertion on the strict positivity/negativity we  use \eqref{eq:114a} and the fact that \ $\Wej(z)\nu=0$ only when $\nu=0$, which follows from Proposition \ref{prop-ltwosol} \ref{eq:ltwosol:3}.

Now, apply \eqref{eq:114a} to $\nu \in \{e_1, e_2, \ldots, e_d \}$. Summing  them up we get  
\[
    \sum_{i=1}^d
    \norm{U(z) e_i}_{[0, +\infty]}^2
    =
    \frac{1}{\Im z}
    \sum_{i=1}^d 
    \big\langle \big(\Im  \Wej(z) \big) e_i, e_i \big\rangle
    =
    \frac{1}{\Im z} \tr \big(\Im \Wej(z) \big).
\]
By Proposition~\ref{prop:18} \ref{prop:18:1} and \eqref{eq:65}   we get
\[
    \sum_{i=1}^d \norm{U(z) e_i}_{[0, +\infty]}^2 =
    \sum_{k=0}^{+\infty} \sum_{i=1}^d \norm{(U(z) e_i)_k}_{\CC^d}^2 
    =
    \sum_{k=0}^{+\infty} \HSnorm{(U(z))_k}^2
    \geq
    \sum_{k=0}^{+\infty} \norm{(U(z))_k}^2 = \norm{U(z)}_{[0,+\infty]}^2.
\]
Hence
\[
    \norm{U(z)}_{[0, +\infty]}^2 \leq 
    \frac{1}{\Im z}
    \tr \big(\Im \Wej(z)\big).
\]
from which the result follows.
\end{proof}

Now we present also a result being a stronger version of \cite[Proposition 3]{block2018}. Let us stress that we {\bf do not assume  the   self-adjointness of $J$}  at this moment.

For $z \in \CC$   denote 
\beql{GEVl2}
    \GEVld{z} := \GEVze{z} \cap \ldnz{\CC^d}.
\eeq
Recall that by Fact~\ref{fact:ini-iso} we have 
$$
\dim \left( \GEVze{z} \right) = 2d, 
$$
and let us think about the dimension of its subspace $\GEVld{z}$. To find it in some cases, consider also  $\EV{z}$  --- the eigenspace for $J$ and $z$, which in the case of arbitrary $z \in \CC$ is  defined by
\[
    \EV{z} := \{ u \in \Dom(J): J u = z u \}
\] 
(and  so, it is just the trivial zero space if $z$ is not an eigenvalue of $J$). Since $J$ is the maximal block Jacobi  operator
(see \eqref{maxJdom}), we have
\beql{EVeq}
    \EV{z} = \{ u \in \GEVld{z} : ((J-z)u)_0 = 0 \}.
\eeq
Indeed, the above equality follows directly from
\[
    \GEVld{z} \subset \Dom(J),
\]
which holds, because for 
$u \in \GEVld{z}$ we have $u \in \ldnz{\CC^d}$, so also $z u \in \ldnz{\CC^d}$, but  $\calJ u$ and $z u$ differ at most at the zero term, hence $\calJ u \in \ldnz{\CC^d}$.

\begin{theorem}
    \label{duzo-malo-wl2}
If $z \in \CC \setminus \si_p(J)$, then $\dim (\GEVld{z}) \leq d$. If, moreover, $z \in \CC \setminus \si(J)$, then $\dim (\GEVld{z}) = d$.
\end{theorem}

\begin{proof}
Consider first an arbitrary $z \in \CC$. 
and define $\Psi: \GEVld{z} \strz \CC^d$ \  by
\[
    \Psi(u) := \left( (J - z\Id) u \right)_0, \quad u \in \GEVld{z}.
\]
It is a linear transformation and by \eqref{EVeq}
$\Ker \Psi = \EV{z}$. So, by the standard linear algebra result,
\[
    \dim \left( \GEVld{z} \right) = 
    \dim \left( \EV{z} \right) + \dim \left( \Ran \Psi \right).
\]
Hence, if $z \in \CC \setminus \si_p(J)$, then $\dim (\GEVld{z}) = \dim (\Ran \Psi) \leq d$.
But, if moreover $z \in \CC \setminus \si(J)$, then
$\Ran (J - z\Id) = \ldnz{\CC^d}$, so in particular, for any $\nu \in \CC^d$ we have $\de_{0}(\nu) \in \Ran(J - z\Id)$. 
Therefore, for some $u \in \Dom(J)$ 
\[
    J u - z u = \de_{0}(\nu),
\]
thus $u \in \GEVld{z}$ and $\Psi(u) = \nu$ for such $u$. So, $\Ran \Psi = \CC^d$ and $\dim (\GEVld{z}) = \dim (\Ran \Psi) = d$.
\end{proof}

This result gives in particular \  $\dim (\GEVld{z}) = d$,
\ when $J$ is s.a. and $z \in \CC \setminus \RR$. On the other hand, one can easily see that for such $z$ 
\[
    \GEVld{z} = \lin \{ u^{(j)}(z): j=1,\ldots, d \},
\]
where $u^{(j)}(z)$ are defined by \eqref{ujz}, and they are just the successive $j$th column sequences ($j=1,\ldots, d$) of the matrix sequence $\tilde{U}(z)$, being the restriction to $\Nz$ of the  
Weyl matrix  solution $U(z)$ for $J$ and $z$  
(see \eqref{def-U(z)}).

\subsection{The Cauchy transform of the spectral matrix measure and the matrix Weyl function}

Assume  here that  $J$ is s.a. \ 
The Cauchy transform  of the spectral matrix measure $M:=E_{J,\vec{\varphi}}$  from \eqref{eq:IV:2.1} of $J$ is  defined as \ \ 
$\Caj:\CC \setminus\RR\strz\Mdc$, with

\beql{Cau-M}
    \Caj(z): = \int_{\RR} \frac{1}{\la - z} \ud M(\la), \quad z \in \CC \setminus\RR,
\eeq
where the above integral is understood in the sense of \eqref{intscalvecmea}.
Consequently,  \eqref{Cau-M}  means just
\beql{Cau-Mij}
    \Caj(z): = 
    \left( 
        \int_{\RR} \frac{1}{\la - z} \ud M_{i,j}(\la) 
    \right)_{i,j=1,\ldots,d\,\mbox{,}}
\eeq
with
\[
    M_{i,j} = E_{J, \varphi_j, \varphi_i}, \quad i,j=1,\ldots,d,
\]
and  $\varphi_j$  given by  \eqref{fix-fincyc}. 

Hence, by spectral calculus for s.a. operators, for   
$z \in \CC \setminus \RR$ we get
\beql{Cau-resJ}
    \Caj(z)
    := 
    \left(
        \int_{\RR} \frac{1}{\la - z}\ud E_{J, \varphi_j, \varphi_i}  (\la)
    \right)_{i,j=1,\ldots,d} 
    = 
    \left(
        \big\langle (J - z \Id)^{-1} \delta_0(e_j), \delta_0(e_i) \big\rangle_{\el^2}         
    \right)_{i,j=1,\ldots,d\,\mbox{.}}
\eeq
Now, by \eqref{ujz}, \eqref{l2scal} and \eqref{def-U(z)}  for any $i,j=1,\ldots,d$
\[
    \left( \Caj(z) \right)_{i,j} = 
    \big\langle u^{(j)}(z), \delta_0(e_i) \big\rangle_{\el^2}  = 
    \big\langle (u^{(j)}(z))_0, e_i \big\rangle_{\CC^d} = 
    \left(U_0(z)\right)_{i,j}
\]
for 
$z \in \CC \setminus \RR$, where $u^{(j)}(z)$ is given by \eqref{ujz} and $U(z)$ is  Weyl matrix  solution for $J$ and $z$.
Finally, by Proposition~\ref{prop-ltwosol}, we get

\begin{fact}\label{fact-WC}
If $J$ is  s.a., then the Cauchy transform  of the spectral matrix measure of $J$ is equal to the matrix Weyl function for $J$ and moreover, \  for any  $z \in \CC \setminus \RR$
\[
    \Caj(z) = 
    \Wej(z) = 
    U_0(z) = 
    \Big(
        \big\langle (J - z \Id)^{-1} \delta_0(e_j), \delta_0(e_i) \big\rangle_{\el^2}          
    \Big)_{i,j=1,\ldots,d}.
\]
\end{fact}

\subsection{The boundary limits of the matrix Weyl function and the properties of the spectral trace measure} \label{sec:boundaryWeyl}

Assume here  the self-adjointness of $J$, as before, and let us use the notation from the previous subsection, including  \  $M := E_{J,\vec{\varphi}}$.

Fact~\ref{fact-WC} \ implies that $W$  is  a holomorphic matrix-valued function. 

Moreover, $\Im W$ is a strictly positive matrix on $\CC_+$ from Theorem \ref {prop:9ogolniejszy}.  Thus  we see that  the restriction of $W$ to  $\CC_+$ is a matrix Herglotz function. Denote 
\[
    \BLW := 
    \Big\{
        \la \in \RR: \lim_{\epsilon \to 0^+} \Wej(\la + i\epsilon) \  \mbox{exists}\footnote{As the  limit in $\Mdc$, i.e., in particular the limit must belong to $\Mdc$.}
    \Big\},
\]
and
\[
    \Wej(\la + i0) := 
    \lim_{\epsilon \to 0^+} \Wej(\la + i\epsilon), \quad 
    \la \in \BLW,
\]
and define the following sets:
\begin{align}
    \label{Sacr}
    \Sacr &:= 
    \big\{ \la \in \BLW : \rank \big( \Im \Wej(\la + i0) \big) = r \big\}, 
    \quad 1 \leq r \leq d, \\
    \label{Sac}
    \Sac &:= 
    \bigcup_{r=1}^d \Sacr\,, \\
    \label{Ssi} 
    \Ssing &:= 
    \Big\{ 
        \la \in \RR : 
        \lim_{\epsilon \to 0^+} \Im \big( \tr \Wej(\la+i\epsilon) \big) = +\infty 
    \Big\}.
\end{align}
In particular it follows that  
\beql{Ssi'}
    \Ssing \subset \RR \setminus \BLW. 
\eeq 

Referring to \eqref{Ssi}, it is worth to note that for any $A \in \Mdc$ 
\[
    \tr(\Im A) = \Im(\tr A).
\]
It is important that \eqref{Sac} means simply 
\beql{Sac'}
    \Sac = \big\{ \la \in \BLW : \Im \Wej(\la + i0) \neq 0 \big\}.
\eeq

Let us recall now the crucial  result, joining   the above defined sets with  properties of \ $M_{\mathrm{ac}}$ \ and  \ $M_{\mathrm{sing}}$ \  --- the a.c. and the sing. parts  \ 
of the spectral matrix measure $M$ of $J$ w.r.t. the  Lebesgue measure $|\cdot|$  on $\Bor(\RR)$ \   (see Fact~\ref{acscM}).
This theorem is obtained just as the direct use of the abstract result 
\cite[Theorem 6.1]{Gesztesy2000} to the spectral  matrix measure $M$.

Define  $D: \BLW \strz\Mdc$  by 
\beql{eq:Mac:a}
    D(\la) := \frac{1}{\pi} \Im \Wej(\la + i0), \qquad \la \in \BLW. 
\eeq

\begin{theorem}\label{ThdensM}
\hspace{0.1 ex}
\begin{enumerate}[label=(\roman*)]
    \item \label{ThdensM:1}
    $\Ssing$  is a support of $\Msing$;
    
    \item \label{ThdensM:2}
    $|\RR\setminus \BLW|=0$;
    
    \item \label{ThdensM:3}
    $D$ is a density of $\Mac$ on $\BLW$ w.r.t. \ $|\cdot|$. 
\end{enumerate}
\end{theorem}

So, this theorem  shows that controlling of the  boundary limits of the matrix Weyl function allows to get a lot of detailed information about the spectral matrix measure of $J$ and  of its sing. and a.c. parts. 

Combining the theorem with Lemma~\ref{techLem} we get: 
 
\begin{proposition} \label{thm:1}
  $\Sac$ is a minimal support of $ (\tr_{M})_{\mathrm{ac}}$ with respect to  $|\cdot|$. Moreover $\Ssing$ is a support of $(\tr_{M})_{\mathrm{sing}}$ \ and \ $|\Ssing| = 0$. Moreover, \ 
  $\Sac\cup\delta$ is also a minimal support of $ (\tr_{M})_{\mathrm{ac}}$ with respect to  $|\cdot|$ for any Borel $\delta\subset\RR$ with \ $|\delta|=0$.
\end{proposition}
\begin{proof}
We use Lemma~\ref{techLem} taking:
\[
    \nu := |\cdot| \ \  \mbox{so, also \  $\Omega:=\RR$ \ and \ $\frakM:=\Bor(\RR)$}, \ \ \ 
    S_a := \Sac\,, \ \ \ 
    S_s := \Ssing\,, \ \ \ 
    F := \restr{D}{\Sac}
\]
and by \eqref{Ssi'} with \eqref{Sac'}  we see that $S_a \cap S_s = \emptyset$, that is, the assumption~\ref{techLem:2} of  the lemma holds.
Now,  \ref{ThdensM:2} of Theorem~\ref{ThdensM} shows that $\BLW$ is a support of $\Mac$, since this  matrix measure, by definition,  is a.c. w.r.t. $|\cdot|$. But \ref{ThdensM:3} of this  theorem with \eqref{Sac'} and  \eqref{eq:Mac:a} mean, that 
the restriction of $\Mac$ to $\BLW \setminus \Sac$ is the zero matrix measure. Therefore $\Sac$ is also a support of $\Mac$. This together with \ref{ThdensM:1} of  Theorem~\ref{ThdensM} prove that the assumption~\ref{techLem:1} of  the lemma holds. The assumption~\ref{techLem:3} also holds, by the fact that $\Sac \subset \BLW$ and by \ref{ThdensM:3} of Theorem~\ref{ThdensM}, again.  And  \ref{techLem:4} of the lemma is obvious by \eqref{Sac'}. Therefore Lemma~\ref{techLem} yields  the first two  assertions  and   $|\Ssing| = 0$ by \eqref{Ssi'} with \ref{ThdensM:3} of the theorem. The last assertion is obvious just by the definition of minimal support.
\end{proof}

\subsection{The boundary limits and spectral consequences for $J$}

Assume the  self-adjointness of $J$, and let us hold the notation  as above.

Here  we ``translate'' Theorem~\ref{ThdensM} into the spectral operator  language via 
Proposition~\ref{thm:1}.  Namely, we show:  

\begin{theorem}
\label{cor:1}
Suppose that $J$ is self-adjoint and $G \in \Bor(\RR)$.
\begin{enumerate}[label=(\roman*)]
    \item \label{cor:1:1}
    If $G \subset \RR \setminus \Ssing$, then $J$ is absolutely continuous in $G$.
    
    \item \label{cor:1:2}
    If $G \subset \Sac\cup(\RR\setminus(\BLW\cup\Ssing))$, then $J$ is absolutely continuous in $G$ and \ $\clleb{G} \subset \sigmaAC(J)$. So, if moreover $G$ is open, or if $G$ is a sum of an arbitrary family of connected non-singletons in $\RR$, then  $\cl{G} \subset \sigmaAC(J)$.
\end{enumerate}
\end{theorem}
\begin{proof}
First of all, by Proposition~\ref{prop-fincyc}, $J$ is s.a.  finitely-cyclic operator, with 
$\vec{\varphi}$  being a cyclic system for $J$ and with  
$M$ being  the  spectral matrix measure of  $J$
and $\vec{\varphi}$.
Hence the initial  assumptions of  \cite[Theorem C.2]{Moszynski2022} hold for $J$. So, using its assertion  (2),  we obtain  our~\ref{cor:1:1}, because, by Proposition~\ref{thm:1}, \ \ $(\tr_{M})_{\mathrm{sing}}(G) = 0$ for $G \subset \RR \setminus \Ssing$. 

By Theorem~\ref{ThdensM} \ref{ThdensM:3} we get $|\RR\setminus(\BLW\cup\Ssing)|=0$.
Hence, $\Sac \cup \big( \RR \setminus (\BLW \cup \Ssing) \big)$ is a minimal support of \ $(\tr_{M})_{\mathrm{ac}}$ with respect to  $|\cdot|$ \ by Proposition~\ref{thm:1}. Moreover $(\tr_{M})_{\mathrm{sing}}(G)=0$, again  because $\Sac \cup \big( \RR \setminus (\BLW \cup \Ssing) \big) \subset \RR \setminus \Ssing$ and 
$\Ssing$ is a support of $(\tr_{M})_{\mathrm{sing}}$ by  Theorem~\ref{ThdensM} \ref{ThdensM:1}.

Now, by  the assertion  (3) of \cite[Theorem C.2]{Moszynski2022},
the assertion for the special kinds of $G$ follows from the property $\clleb{G} = \cl{G}$, which holds for those $G$ (see, e.g., \cite[Fact C.3]{Moszynski2022}).   
\end{proof}

\section{On some ideas of nonsubordinacy for the block case} \label{sec:Subord}

\subsection{The  nonsubordinacy and the vector nonsubordinacy} \label{sec:m-v-nonsubodinacy}

Let us define a   notion, which seems quite  natural from the context of the crucial notion of {\em subordinated solutions}  from  Gilbert--Pearson--Khan 
subordination theory (see \cite{Gilbert1987, Khan1992}),  concerning   the $d=1$ case.
Recall that the ``interpolated'' semi-norms $\norm{\cdot}_{[0,t]}$ were 
introduced here in Section~\ref{J-Lsemi}.

\begin{definition} \label{def:vnonsub}
We say that {\em $J$ satisfies vector nonsubordinacy (condition) for $\la \in \RR$} \ iff\footnote{Note that we consider here  the function given  by the  fraction $\frac{\norm{u}_{[0,t]}^2}{\norm{v}_{[0,t]}^2}$ for $t>0$ only, and the denominator is positive since here sequences are  not the zero sequence and they belong to $\GGEVze$; Similarly for the definition below.} for each pair of  non-zero $u,v \in \GEV{ \la}$ 
\begin{equation} \label{eq:121}
    \liminf_{t \to +\infty} 
    \frac{\norm{u}_{[0,t]}}{\norm{v}_{[0,t]}}  < +\infty.\hspace{0.1em}\hspace{0.3em} 
\end{equation}
\end{definition}

We can also consider ''another'' (soon...) notion, using the analogical  matrix (emgev) terms formulation.

\begin{definition} \label{def:nonsub}
$J$ satisfies \emph{nonsubordinacy} condition for $\la \in \RR$ \ iff \ for each pair of non-zero $U,V \in \MGEV{ \la}$
\begin{equation} \label{eq:122}
    \liminf_{t \to +\infty} 
    \frac{\norm{U}_{[0,t]}}{\norm{V}_{[0,t]}} < +\infty.
\end{equation}
\end{definition}
Note here that the choice of  ``liminf-s over $t>0$'', instead of more  original Khan and Pearson's like ``liminf-s over $n\in\NN$'', does not matter in both definitions, i.e., in both of them it  is equivalent to this ''over $n\in\NN$''. One direction of the implication follows directly from the definition of $\liminf$, and the other can be immediately obtained by Corollary~\ref{szac-iloraznort-norn}.
As we shall see soon, also the distinction  between these  two kinds ``vector nonsubordinacy'' / ``nonsubordinacy'' notions is not very important here.

We can use the symmetry w.r.t. $u$ and $v$ or $U$ and $V$, respectively, and we get:

\begin{fact}\label{factsym}
$J$ satisfies {vector nonsubordinacy}  for $\la \in \RR$ \ iff \ for each pair of  non-zero $u,v \in \GEV{ \la}$ 
\begin{equation} \label{eq:121'}
    \limsup_{t \to +\infty} 
    \frac{\norm{u}_{[0,t]}}{\norm{v}_{[0,t]}}  >0. 
\end{equation}
And analogically in the   nonsubordinacy case.
\end{fact}

As we already  announced, the  above  two definitions are equivalent, so we shall rather use only the name 'nonsubordinacy'' and not ''vector nonsubordinacy''.

\begin{proposition} \label{prop:17}
Let $\la \in \RR$. $J$ satisfies  nonsubordinacy  for $\la$ \ iff \  it satisfies vector nonsubordinacy for $\la$.
\end{proposition}
\begin{proof}
$(\Rightarrow)$ Take any non-zero $u,v \in \GEV{ \la}$ and view them as a sequence of column vectors. Let us define 
\[
    U_n := E^{u_n}, \quad 
    V_n := E^{v_n}, \quad n \geq -1,
\]
cf. \eqref{def:Ev}, i.e. the first column of $U_n$ is equal to the column vector $u_n$ and the rest is zero, analogously for $V_n$. By Fact~\ref{factsolve-ma-ve} both $U,V \in \MGEV{ \la}$. By Proposition~\ref{prop:18} for any $n \geq -1$ we have $\norm{U_n} = \norm{u_n}$ and $\norm{V_n} = \norm{v_n}$. Consequently, $\norm{U}_{[0,t]} = \norm{u}_{[0,t]}$ and $\norm{V}_{[0,t]} = \norm{v}_{[0,t]}$ for any $t \geq 0$. Thus, the condition \eqref{eq:122} implies \eqref{eq:121}.

$(\Leftarrow)$ Let us observe that for any $X \in \MGEV{\la}$ and any $w \in \CC^d$ such that $\norm{w} = 1$ we have
\begin{equation} \label{eq:136}
    \norm{X w}_{[0,t]}^2 
    \leq
    \norm{X}_{[0,t]}^2
    \leq 
    \sum_{i=1}^d \norm{X e_i}_{[0,t]}^2, \quad t > 0.
\end{equation}
Let $U,V \in \MGEV{\la}$ be non-zero. Then there exists $w \in \CC^d$ such that $\norm{w}=1$ and $V w$ is non-zero. Then by \eqref{eq:136} we have
\begin{equation} \label{eq:137}
    \frac{\norm{U}_{[0,t]}^2}{\norm{V}_{[0,t]}^2} 
    \leq
    \sum_{i=1}^d \frac{\norm{U e_i}_{[0,t]}^2}{\norm{V w}_{[0,t]}^2}, \quad t > 0.
\end{equation}
Now, by Fact~\ref{factsolve-ma-ve}, $U e_i \in \GEV{\la}$ for $i=1,\ldots,d$ and $V w \in \GEV{\la}$. Thus, by \eqref{eq:121} we get
\[
    \liminf_{t \to +\infty}
    \frac{\norm{U e_i}_{[0,t]}^2}{\norm{V w}_{[0,t]}^2} < +\infty, \quad i=1,\ldots,d,
\]
which together with \eqref{eq:137} implies \eqref{eq:122}.
\end{proof}

\subsection{Some ``fast'' spectral  consequence of  nonsubordinacy} \label{sec: spectr-nonsub}

In fact, this ``negative form'' of subordination type  notions is not only ``natural'', as it was mentioned before, but it is also related to some spectral results for $J$ in the block case which are  a bit like the absolute continuity. 

 Recall that $\GEVld{\la}$ was defined in \eqref{GEVl2}.

\begin{theorem} \label{prop:16}
Suppose that  $J$ is self-adjoint and it satisfies  nonsubordinacy for some $\la \in \RR$. Then \ $\dim (\GEVld{\la}) = 0$ \ and \ $\la \in \sigma(J) \setminus \sigmaP(J)$.
\end{theorem}
\begin{proof}
 Let $u,v \in \GEV{ \la}$ 
be non-zero. By Fact~\ref{factsym}
there exists a constant $c>0$ and a sequence $(t_k)_{k \in \NN}$ tending to $+\infty$ such that
\[
    \norm{u}_{[0,t_k]} \geq c \norm{v}_{[0,t_k]}.
\]
Consequently,  taking the limit we get
\begin{equation} \label{eq:96}
    \frac1c\norm{u}_{[0,+\infty]} \geq  \norm{v}_{[0,+\infty]}.
\end{equation}
Thus, if there existed a non-trivial $u \in \ldnz{\CC^d}$, then all $v$ would be also in $\ldnz{\CC^d}$. But then  the operator $J$ would not be self-adjoint, by \cite[Theorem 1.3]{Dyukarev2020}. 
Thus each non-zero $u\in \GEV{ \la}$ is not square-summable. Therefore, by \eqref{EVeq} we get $\lambda \notin \sigmaP(A)$. Finally,  Proposition~\ref{duzo-malo-wl2} yields the  last assertion.
\end{proof}

\subsection{On further nonsubordination ideas}

However, this concept of nonsubordinacy turns out to be rather insufficient for the purposes of true absolute continuity. In our  parallel paper \cite{BarrierNonsubordinacy} we develop this idea  deeper and we consider more sophisticated  notion  called {\em barrier nonsubordinacy}. This new condition  can already guarantee the spectral absolute continuity property for the block case of  $J$. And the proof of it is based on  arguments somewhat analogical to those used in the proofs of subordination theory for  $d=1$, including the generalization for the block case of  Weyl theory  presented in the previous sections.

Let us only mention here, that  
the barrier nonsubordinacy 
 controls two things:
  a bound on  \ $\frac{\norm{U}_{[0,t]}}{\norm{V}_{[0,t]}}$ \  for any fixed ``large'' $t$ and  $\la\in G$, but joint for all such $U,V \in \MGEV{ \la}$ which are ``normalized'' in a  certain  sense, and also 
   ''the size'' of the above bound as a function of $t$.

\appendix  
\section{Vector  and  matrix measures --- selected basic notions} \label{vecmacmeas}

For  self-consistency and some self-sufficiency of the paper  we  collect here  selected  definitions of   some basic notions and some properties  related to matrix measures and --- more generally ---  vector measures.
We omit here the more sophisticated construction of the appropriate $L^2$-type space for the matrix measure, referring the reader to the literature  (see, e.g., \cite[Section 8]{Weidmann1987} or \cite{Moszynski2022}\footnote{The second position contains the detailed definition of the Hilbert space  $L^2(M)$ for matrix measures and the details of the abstract   spectral theory for finitely cyclic s.a. operators, based on the matrix measure approach and multiplication by a function operators in $L^2(M)$ spaces.})

 We start from the general definition of vector measure.
 
 Consider a set $\Omega$ with  $\frakM$ --- a $\si$-algebra of subsets of $\Omega$ and 
  a certain norm space  $X$. Let $V:\frakM\strz X$ 
 
\begin{definition} \label{vectormeasure}
\emph{$V$ is a vector measure (in $X$)} iff
$V$ is countably additive in the norm sense in $X$.
\end{definition}

Now, let us consider a  special case  $X:=\Mdc$ for some $d\in\NN$ (with a standard norm, say). 
 And   let  \ \ 
$M : \frakM \strz M_d(\CC)$.

\begin{definition} \label{matrmeas}
\emph{$M$ is a ($d \times d$) matrix measure} iff
\begin{enumerate}[label={(\alph*)}]
    \item \label{matrmeas:1}
    $M$ is a vector measure;
    
    \item \label{matrmeas:2}
    $M(\omega) \geq 0$\  for any $\omega \in \frakM$.\footnote{So, in particular $M(\omega)$ is s.a..}
\end{enumerate}
\end{definition}
So, in particular, each matrix measure is a  vector measure, 
but despite its name and  due to the extra  non-negativity property~\ref{matrmeas:2}, matrix measure  is ``much more'' than a vector measure  in $M_d(\CC)$.

The above  $\Omega$,   $\frakM$ and $d$ are ``fixed'' below.

We have  to precise here  some terminology   (choosing it from various versions  in  literature)   and to recall several basic facts related to vector measures, matrix measures  and measures. Recall that some ''measure'' terminology was already fixed  in  Introduction --- see footnote \ref{foot-meas} page \pageref{foot-meas}.

 Suppose that  $\nu$ is a measure on $\frakM$ and  $V:\frakM\strz X$ is a vector measure, where  $X$ is a norm space. 
 
 In several cases below we will also need to  assume additionally that $X=\CC^k$ for some $k$, including  possible  obvious identifications, as,  e.g., $\Mdc\equiv\CC^{(d^2)}$, to provide the  clear and standard sense of the integral and of  $\calL^1_{X}(\nu)$ functions (see the appropriate part of Section~\ref{L1X}). 
 
For any $G \in \frakM$ denote $\frakM_{G} := \{ \omega \in \frakM : \om \subset G \}$.
Surely $\frakM_{G}$ is a $\sigma$-algebra of subsets of $G$ and  
$\restr{V}{\frakM_{G}}$ is a vector measure on $\frakM_{G}$ ("on $G$"). We denote it by $V_{G}$, i.e.
\[
	V_{G} := \restr{V}{\frakM_{G}},
\]
and we call it \emph{the restriction of $V$ to $G$}. 
Analogous situation (and notation, and  terminology) is well known and  valid  here for  measures.

Suppose that $H:\Omega\strz X=\CC^k$ is a measurable function w.r.t. $\frakM$. If, moreover, 
$H\in\calL^1_{X}(\nu)$, then we define a new function from $\frakM$ into $X$ by the formula:
\beql{densdmeas}
    \int_{\om} H \ud \nu, \qquad \om \in \frakM.
\eeq
It is obviously a vector measure, and we denote it by 
\[
    H \ud \nu,
\]
analogously as in the case of  measures (when $H$ should be a scalar non-negative measurable  function, instead of $H\in\calL^1_X(\nu)$). 
But in our main case  $X=\Mdc$  \ (identified with $\CC^{(d^2)}$) 
the situation is somewhat similar and one easily checks the following  result.
\begin{fact} \label{posden=mm}
 If \ 
$H \in \calL^1_X(\nu)$ \ and \  $H(t) \geq 0$ \ for $\nu$-a.e. $t\in \Omega$, then  the vector measure 
$
	H \ud \nu
$
is a matrix measure.   
\end{fact}

\begin{definition} \label{defilong}
If a  vector measure $V$ is such, that  
$V = H \ud \nu$ with  some $H \in \calL^1_{X}(\nu)$, then we call $H$  {\em the density\footnote{However, it can be  not unique, as a function from  $\calL^1_{X}(\nu)$.} of \ $V$ with respect to $\nu$}. 

Let $G  \in \frakM$.

\bde
\item[the density  on $G$ w.r.t.] 
{\em The density of $V$ on $G$ w.r.t. $\nu$} means:  any   density of $V_G$ w.r.t. $\nu_G$;

\item[a support]
\emph{$G$ is a support of \ $V$} \ iff \  $V_{\Omega\setminus G}$ is the  zero vector measure  (on $\frakM_{\Omega\setminus G}$)\footnote{\label{fosuveme} Generally, it is not sufficient here (contrary to measures) that $V(\Omega\setminus G)=0$ because the property of having zero measure is not inheritable  into  measurable subsets. However, for matrix measures it {\bf is} sufficient by non-negativity.};

\item[a minimal  support  w.r.t.]
\emph{$G$ is a minimal  support of \ $V$ w.r.t. $\nu$} \ iff \  
$G$ is a   support of \ $V$
and for any support $G'$  of \ $V$ included in $G$
\[
	\nu(G\setminus G')=0;
\]

\item[a.c. w.r.t.]
\emph{$ V$ is absolutely continuous (abbrev.: a.c.) w.r.t. $\nu$} \  iff \ for any $\om \in \frakM$ \ if $\nu(\om) = 0$, then  $V(\om) = 0\,$;

\item[sing. w.r.t.]
\emph{$V$ is singular (abbrev.: sing.) w.r.t. $\nu$} iff  there exists such a support $S\in \frakM$ of $V$ that $\nu(S) = 0$;

\item[the a.c./sing. part  w.r.t.]
If $V_1, V_2:\frakM\strz X$ are two vector measures, such that
\begin{enumerate}[label=(\roman*)]
    \item $V_1$ is a.c. w.r.t. $\nu$ and $V_1$ is sing. w.r.t. $\nu$,
    \item $V = V_1 + V_2$,
\end{enumerate}
then we call $V_1$ \  {\em the a.c. part of $V$ w.r.t. $\nu$} and we denote it by 
\[
    V_{\mathrm{ac}, \nu},
\]
and we call $V_2$ \  {\em the sing. part of $V$ w.r.t. $\nu$}, \  and we denote it by 
\[
    V_{\mathrm{sing},\nu}.
\]
Note here, that above two notions are well(uniquely)-defined, since the above decomposition, if exists, is unique\footnote{Because any  linear combination of vector measures being a.c. (sing.) w.r.t. $\nu$ is also a.c. (sing.) w.r.t. $\nu$; and a vector measure which is both a.c. and sing. w.r.t. $\nu$ is the zero measure.};

\item[a.c. on $G$ w.r.t.]
\emph{$ V$ is a.c. on $G$ w.r.t. $\nu$} \  iff \ 
{$ V_G$ is a.c. w.r.t. $\nu_G$}; 

\item[sing. on $G$ w.r.t.]
\emph{$ V$ is sing. on $G$ w.r.t. $\nu$} \  iff \ 
{$ V_G$ is sing. w.r.t. $\nu_G$.} 

\ede

We adopt  all the above definitions and names also for any  measure $\mu$ instead of a vector measure $V$ (recall that  measure may  be  not a vector measure) just by interchanging the symbols $\mu$ and $V$, including the notation
\[
    \mu_{\mathrm{ac},\nu}, \qquad \mu_{\mathrm{sing},\nu},
\]
however they are most commonly known in the case of measures.  
 
We adopt here also  the convention, that the part ``w.r.t. $\nu$'', as well as ``$,\nu$'' in the appropriate symbols, as e.g.,   $V_{\mathrm{ac},\nu}$, $V_{\mathrm{sing},\nu}$, can be omitted   to   shorten the notation  to,  e.g.,
\[
    V_{\mathrm{ac}}, \ V_{\mathrm{sing}}, \ 
    \mu_{\mathrm{ac}}, \ \mu_{\mathrm{sing}}
\]
etc, in the case when 
$\Omega=\RR$, \ $\frakM=\Bor(\RR)$ and $\nu$ is the Lebesgue measure $|\cdot|$.
\end{definition}

For a $d \times d$ matrix measure $M$ and $i,j=1,\ldots, d$ let us define $M_{ij}: \frakM \to \CC$ by
\begin{equation}
	\label{eq:III:1.1}
	M_{ij}(\omega) := \big( M(\omega) \big)_{ij}\, , \quad 
	\omega \in \frakM.
\end{equation}
By the point~\ref{matrmeas:1} of Definition~\ref{matrmeas}, each of the  $M_{ij}$ is a complex measure on $\frakM $. Moreover, 
non-negativity from \ref{matrmeas:2} of a matrix means   also its self-adjointness, so  we have
\beql{Msym}
    M_{j,i}=\sprz{M_{i,j}} \quad           i,j=1,\ldots,d.
\eeq

By non-negativity, defining  
 $
 \tr_M: \frakM \to \CC
 $
by the formula 
\beql{trM}
\tr_M(\om):=\tr(M(\om)), \quad \om \in\frakM,
\eeq
we get  in fact a finite  measure $\tr_M$, called  {\em    trace measure of $M$.}   
This ``classical'' measure is much a  simpler mathematical object than the matrix measure $M$, but it contains a lot of important information on $M$. 

The results below are proved in \cite[Section III.1]{Moszynski2022}.

\begin{fact} \label{prop:III:3}
For   $M$ being  a matrix measure as above:
\begin{enumerate}[label=(\roman*)]
    \item \label{prop:III:3:1}
    $M$, as well as  each $M_{ij}$ for $i,j=1,\ldots,d$, are absolutely continuous with respect to $\tr_M$;
    
    \item \label{prop:III:3:2}
    \beql{Mifftr}
        M(\om)=0 \iff \tr_M(\om)=0, \quad\qquad \om \in \frakM;
    \eeq
    
    \item \label{prop:III:3:3}
    \beql{Mleqtr}
        \ 0 \leq M(\omega) \leq \tr_M(\omega) \Id \quad\qquad \om \in \frakM;
    \eeq
   
    \item \label{prop:III:3:4} 
    There exists a density $D: \Omega \strz \Mdc$
    of $M$ w.r.t. $\tr_M$, such that 
    for any $t \in \Omega$ \ 
    \[
        0 \leq D(t) \leq \Id\ ,\footnote{In particular, $D(t)$  is self-adjoint.}  \qquad  t \in \Omega.
    \]
\end{enumerate}
\end{fact}

 Note that each density $D$ of $M$  w.r.t. $\tr_M$ is determined only up to   $\tr_M$-a.e. equality, and each one is  called {\em the trace density of $M$}.  The set of all the densities of $M$ w.r.t. $\tr_M$  is denoted here by $\DD_M$, and by $\DD^{\bullet}_M$ we denote the set of those $D$, which satisfy conditions from  the  point~\ref{prop:III:3:4} above.

By the definition of density we have 
\beql{eq:III:1.5}
	 M(\omega) = 
	\int_\omega D \ud \tr_M, \qquad \omega \in \frakM, \ D\in \DD_M.
\eeq

We assume here, {\bf to the end if this section}, that 
  $M : \frakM \strz M_d(\CC)$ is a matrix measure and $\nu :\frakM \strz [0,+\infty]$ is a $\si$-finite measure.

There exist several ways  to get a  decomposition of {\bf some} vector measures $V$ into its a.c. and sing. parts w.r.t. a measure $\nu$.
All they are based somehow  on the Lebesgue--Radon--Nikodym Theorem (see, e.g. \cite{Rudin-RCA}) for a complex measure ``w.r.t. a $\si$-finite measure'' version. We need it only for our  matrix measure $M$
and it will be convenient to make it {\em via} the appropriate decomposition of $\tr_M$ onto parts $(\tr_M)_{\mathrm{ac},\nu}$ and $(\tr_M)_{\mathrm{sing},\nu}$,  which exist by the Lebesgue--Radon--Nikodym Theorem.
So, we just consider any $D\in\DD_M$ and  two matrix measures:
\[
    D \ud(\tr_M)_{\mathrm{ac},\nu}, \qquad   D \ud(\tr_M)_{\mathrm{sing},\nu}
\]
(see notation \eqref{densdmeas})

\begin{fact}\label{acscM}
 The a.c. and the sing.  parts of $M$ w.r.t. $\nu$ exist, and they satisfy
\beql{acsingparts}
    M_{\mathrm{ac},\nu} = D \ud(\tr_M)_{\mathrm{ac},\nu}\,, \qquad 
    M_{\mathrm{sing},\nu} = D \ud(\tr_M)_{\mathrm{sing},\nu}\,,
\eeq
where $D$ is an arbitrary density from $\DD_M$.
Moreover, both $M_{\mathrm{ac},\nu}, \  M_{\mathrm{sing},\nu}$ are matrix measures\footnote{By the definition, they are ``only'' vector measures in $\Mdc$.}, and  
\beql{trac=actr}
    \tr_{M_{\mathrm{ac},\nu}} = (\tr_M)_{\mathrm{ac},\nu}, \qquad 
    \tr_{M_{\mathrm{sing},\nu}} = (\tr_M)_{\mathrm{sing},\nu}.
\eeq
In particular, if  $S\in\frakM$, then   TFCAE:
\begin{itemize}
    \item $S$ is a support (version 2.: minimal support w.r.t. $\nu$) of \ $M_{\mathrm{ac},\nu}$\,;
    \item $S$ is a support (version 2.: minimal support w.r.t. $\nu$) of \ $\tr_{M_{\mathrm{ac},\nu}}$\,;
    \item $S$ is a support (version 2.: minimal support w.r.t. $\nu$) of \ $(\tr_{M})_{\mathrm{ac},\nu}$\,;
\end{itemize}
and analogically for 
$M_{\mathrm{sing},\nu}$, \ $\tr_{M_{\mathrm{sing},\nu}}$, \ $(\tr_{M})_{\mathrm{sing},\nu}$.
\end{fact}
\begin{proof}
Using ``the  short theory'' presented above, one immediately checks  that 
the pair of vector  measures \ 
$Dd(\tr_M)_{\mathrm{ac},\nu}$,
$Dd(\tr_M)_{\mathrm{sing},\nu}$ \ 
satisfies the conditions from  Definition~\ref{defilong} of the 
parts $M_{\mathrm{ac},\nu}$
$M_{\mathrm{sing},\nu}$. So, by the uniqueness of the decomposition, we get
\eqref{acsingparts}. Now, by 
  the non-negativity of $D$ and  by Fact~\ref{posden=mm} we see that $M_{\mathrm{ac},\nu}, \  M_{\mathrm{sing},\nu}$ are matrix measures. 

  To get the assertion \eqref{trac=actr}, observe that the equality \  $M=M_{\mathrm{ac},\nu}+  M_{\mathrm{sing},\nu}$ \ \  yields  \ 
  $\tr_M=\tr_{M_{\mathrm{ac},\nu}}+  \tr_{M_{\mathrm{sing},\nu}}$ \ by the definition of the trace measure.
  But $M_{\mathrm{ac},\nu}, \   M_{\mathrm{sing},\nu}$ are a.c. or, respectively, sing. w.r.t. $\nu$, hence also $\tr_{M_{\mathrm{ac},\nu}}, \   \tr_{M_{\mathrm{sing},\nu}}$ are a.c. or, respectively, sing. w.r.t. $\nu$, just by the use of   the property \eqref{Mifftr} for both  those   matrix measures. So, we get the result just by the definitions of a.c. and sing. parts.
  And the last part follows directly from  \eqref{trac=actr}, \eqref{Mifftr} and by the observation that both  notions: of support, as  well as of minimal support w.r.t. a measure, are determined by zero vector measure  sets, only.
\end{proof}

Now we turn  to a ``technical'' result concerning the notion of the minimal support.

\begin{lemma}\label{techLem}
Consider a matrix measure $M : \frakM \strz M_d(\CC)$, a measure $\nu : \frakM \strz [0,+\infty]$, \ sets   $S_a, S_s\in\frakM$ \ and a { non-negative} function $F:S_a\strz M_d(\CC)$.
If
\begin{enumerate}[label=(\roman*)]
    \item \label{techLem:1}
    \ \ $S_a$ \ is a support of \ $M_{\mathrm{ac},\nu}$ \ and \  $S_s$ \ is a support of \ ${M_{\mathrm{sing},\nu,}}$
    
    \item \label{techLem:2}
    \ \ $S_a\cap S_s=\emptyset$,
    
    \item \label{techLem:3}
    $F$ is a density of \ $M_{\mathrm{ac},\nu}$ \  on \ $S_a$ \ w.r.t. \  $\nu$,
    
    \item \label{techLem:4}
    \ \ $\nu\left(\{t\in S_a: F(t)=0\}\right)=0$,
\end{enumerate}
then $S_s$ is a support of \ $(tr_M)_{\mathrm{sing},\nu}$ \ and \
$S_a$ is a minimal support of \ $(tr_M)_{\mathrm{ac},\nu}$ \ w.r.t. \ $\nu$.
\end{lemma}
\begin{proof}
From Fact~\ref{acscM}
we immediately see that 
 $S_s$ is a support of \ $(\tr_M)_{\mathrm{sing},\nu}$ \ and \
$S_a$ is  support of \ $(\tr_M)_{\mathrm{ac},\nu}$.
To prove the minimality, consider any $S'\subset S_a$ which is also a support of $(\tr_M)_{\mathrm{ac},\nu}$.
Hence, again by Fact~\ref{acscM}
\beql{nr1}
    0 = 
    (\tr_M)_{\mathrm{ac},\nu}(S_a\setminus S') = 
    \tr_{M_{\mathrm{ac},\nu}}(S_a\setminus S').
\eeq
On the other hand, by the assumption~\ref{techLem:3}, using 
$(S_a \setminus S')\subset S_a$
and the non-negativity of  $\tr F(t)$ for any $t\in S_a$, we have
\begin{align*}
    \tr_{M_{\mathrm{ac},\nu}}(S_a\setminus S') = 
    \tr \big( M_{\mathrm{ac},\nu} (S_a \setminus S') \big) = 
    \tr \left( \int_{(S_a\setminus S')} F \ud \nu \right) = 
    \int_{(S_a\setminus S')} \tr F(t) \, \ud \nu(t).
\end{align*}
Thus, by \eqref{nr1} we have $\tr F(t)=0$ \  for $\nu$-a.e. $t\in(S_a\setminus S')$. Moreover, by
Proposition~\ref{prop:18}\ref{prop:18:4}, $\tr F(t)=0 \iff F(t)=0$. So,  by the  assumption~\ref{techLem:4},   we get $\tr F(t) \neq 0$ \ also  for $\nu$-a.e. $t\in(S_a\setminus S')$. Thus 
$\nu(S_a \setminus S') = 0$.    
\end{proof}

At the end of this section let us recall the definition of the integral of the scalar function w.r.t. a vector measure for some simplest case,  but sufficient  for our goals.  Consider a vector measure $V
:\frakM\strz X$, where $X=\CC^k$ (e.g. ---  a matrix measure, with $k=d^2$) and $f:\Omega\strz\CC$ --- a bounded $\frakM$-measurable function. Then for any $s=1,\ldots, k$ the  function $V_s:\frakM\strz \CC$ given by
\[
    V_s(\om):= \big( V(\omega) \big)_s, \qquad \omega\in\frakM,
\]
is a complex measure and therefore  the integral 
\[
    \int_{\Omega} f \ud V_s
\]
is well-defined (see, e.g., \cite[Section I.1]{Diestel1977}). 
Thus,  we define simply:
\beql{intscalvecmea}
    \int_{\Omega} f \ud V:= 
    \left(\int_{\Omega} f \ud V_1, \ldots, \int_{\Omega} f \ud V_k \right)\in \CC^k.
\eeq

\begin{bibliography}{block,jacobi}
	\bibliographystyle{amsplain}
\end{bibliography}
\end{document}